\documentclass{amsart}

\usepackage{kantlipsum} 
\calclayout

\usepackage{amsmath,amssymb,amscd}
\usepackage{mathtools}
\usepackage{graphicx}
\usepackage{wasysym,footnote}
\usepackage{hyperref}
\usepackage[all]{xy}
\input{quivers_HomE.sty}

\newtheorem{theorem}{Theorem}[section]
\newtheorem{lemma}[theorem]{Lemma}
\newtheorem{corollary}[theorem]{Corollary}
\newtheorem{proposition}[theorem]{Proposition}

\theoremstyle{definition}
\newtheorem{definition}[theorem]{Definition}
\newtheorem{example}[theorem]{Example}
\newtheorem{conjecture}[theorem]{Conjecture}

\theoremstyle{remark}
\newtheorem{remark}[theorem]{Remark}
\newtheorem{question}[theorem]{Question}
\newtheorem{problem}[theorem]{Problem}

\numberwithin{equation}{section}

\DeclareMathOperator{\Aut}{Aut}
\DeclareMathOperator{\Coker}{Coker}

\DeclareMathOperator{\E}{E}

\DeclareMathOperator{\Ext}{Ext}
\DeclareMathOperator{\ext}{ext}

\DeclareMathOperator{\Gr}{Gr}
\DeclareMathOperator{\Hom}{Hom}

\DeclareMathOperator{\Ker}{Ker}

\DeclareMathOperator{\PHom}{PHom}

\DeclareMathOperator{\rep}{rep}

\newcommand{\op}[1]{\operatorname{#1}}
\newcommand{\mb}[1]{\mathbb{#1}}
\newcommand{\mc}[1]{\mathcal{#1}}

\renewcommand{\b}[1]{\bold{#1}}

\newcommand{\e}{{\op{e}}}
\newcommand{\g}{{\sf g}}
\newcommand{\N}{{\sf N}}
\newcommand{\V}{{\sf V}}

\newcommand{\Ec}{{\check{\E}}}
\newcommand{\ec}{{\check{\e}}}
\newcommand{\fc}{{\check{f}}}
\newcommand{\tc}{{\check{t}}}
\newcommand{\dc}{{\check{d}}}
\newcommand{\dtc}{{\check{\delta}}}
\newcommand{\dtb}{{\br{\delta}}}

\newcommand{\ep}{{\epsilon}}

\newcommand{\br}[1]{\overline{#1}}
\newcommand{\innerprod}[1]{\langle#1\rangle}
\renewcommand{\ss}[2]{{\rep^{#1}({#2})}}

\newcommand{\dv}{\underline{\dim}}

\newcommand{\cone}[1]{\mb{R}_+\Sigma_{\beta_{#1}}({K_{#1,#1}^2})}

\renewcommand{\cone}{\mb{R}_{\geq 0}}
\newcommand{\yb}{\bar{y}}

\begin{document}
\title{Combinatorics of $F$-polynomials}
\author{Jiarui Fei}
\address{School of Mathematical Sciences, Shanghai Jiao Tong University}
\email{jiarui@sjtu.edu.cn}
\thanks{The author was supported in part by National Science Foundation of China (No. 11971305).}

\subjclass[2010]{Primary 16G20; Secondary 13F60,\ 52B20}

\date{}
\dedicatory{}
\keywords{$F$-polynomial, Quiver Representation, Cluster Algebra, Stabilization Functor, Newton Polytope, Saturation}

\begin{abstract} We use the stabilization functors to study the combinatorial aspects of the $F$-polynomial of a representation of any finite-dimensional basic algebra. 
We characterize the vertices of their Newton polytopes.
We give an explicit formula for the $F$-polynomial restricting to any face of its Newton polytope.
For acyclic quivers, we give a complete description of all facets of the Newton polytope when the representation is general.	
We also prove that the support of the $F$-polynomial is saturated for any rigid representation.
We provide many examples and counterexamples, and pose several conjectures. 
\end{abstract}

\maketitle
\section*{Introduction}
The $F$-polynomial discussed in this paper originated from the theory of cluster algebras \cite{FZ1}.
For every $p\times q$ matrix $B$ whose {\em principal part} is skew-symmetrizable, there is an associated cluster algebra $\mc{C}(B)$,
which is contained in the field of rational functions $\mb{C}(x_1,x_2,\dots,x_q)$.
It is defined by a set of generators, called {\em cluster variables}, that are computed recursively via an operation called
{\em mutation}.
We know from \cite{FZ4} that any such {cluster variable} can be written as 
$$\prod_{i=1}^q x_i^{\g(i)} F(y_1,y_2,\cdots,y_p).$$
Here the polynomial $F$ is called the {\em $F$-polynomial} of the cluster variable,
and $y_i=\prod_{j=1}^q x_j^{B_{i,j}}$.
In what follows we will use $\b{x}^\g$ to denote the monomial $\prod_{i} x_i^{\g(i)}$ and similarly for $\b{y}^\gamma$. 

The $F$-polynomial was later endowed with the representation theoretical meaning through the categorification (e.g. \cite{CC,DWZ2}).
If the cluster algebra is {\em skew-symmetric}, then the above polynomial $F$ is the $F$-polynomial of a {\em reachable} representation of some {\em quiver with potential}.
Roughly speaking, the {\em $F$-polynomial} of a quiver representation $M$ is the generating series of the topological Euler characteristic of the {\em representation Grassmannian} of $M$:
\begin{equation*} F_M(\b{y}) = \sum_{\gamma} \chi(\Gr_\gamma(M)) \b{y}^\gamma. \end{equation*}
	
Many deep results and conjectures of the cluster algebras are related to the $F$-polynomials.
For example, the original positivity conjecture \cite{FZ1} is equivalent to that all coefficients in the above $F$-polynomial are positive.
A large class of cluster algebras admit a {\em generic basis} $\{\b{x}^{\g_M} F_M(\b{y})\}$ where $M$ runs through a certain class of generic representations \cite{P,GLSg}.

Despite its elegant definition and important role, very little research so far has focused on the $F$-polynomial  (rather than a single coefficient $\chi(\Gr_\gamma(M))$), especially the combinatorial aspect.
We recall that the Newton polytope of a multivariate polynomial $\sum_{k} c_k \b{y}^{\gamma_k}\ (c_k\neq 0)$ is the real convex hull of all $\gamma_k$'s.
In this paper, we study the following three natural problems:
\begin{problem}\label{p:a} How to characterize the vertices of the Newton polytope of $F_M$? What are the coefficients corresponding to the vertices? 
\end{problem}

\begin{problem}\label{p:b} How to describe all facets of the Newton polytope of $F_M$? What is the restriction of $F_M$ to a particular facet? 
\end{problem}

\begin{problem}\label{p:c} Is the support of $F_M$ {\em saturated} (to be explained below)? 
\end{problem}

We have a nice solution to Problem \ref{p:a}. The {\em Newton polytope ${\N}(M)$ of a representation} $M$ is by definition the convex hull of
$$\{ \dv L \mid L\hookrightarrow M \}$$
in $\mb{R}^{Q_0}$. 
\begin{theorem}[Theorem \ref{T:onept}] \label{T:intro1} If $\gamma$ is a vertex of ${\N}(M)$ then $\Gr_\gamma(M)$ must be a point.
\end{theorem}
\noindent  This shows in particular that the Newton polytope of $M$ is the same as the usual Newton polytope of the polynomial $F_M$.
The converse does not hold in general but we conjecture at least for the {\em Jacobian algebras} that the converse holds for $M$ being a cokernel of a general presentation.
We have a list of properties for the subrepresentation corresponding to the unique point (Proposition \ref{P:vhoms}).

For a {\em rigid} representation $M$ of an acyclic quiver (that is $M$ satisfying $\Ext^1(M,M)=0$), we have the following characterization of vertices.
\begin{theorem}[Theorem \ref{T:acyclicv}] \label{T:intro2} Let $M$ be an $\alpha$-dimensional rigid representation of an acyclic quiver. Then $\gamma$ is a vertex of ${\N}(M)$ if and only if $\gamma \perp (\alpha -\gamma)$.
\end{theorem}
\noindent Here $\gamma \perp \beta$ means $\Hom(L,N) = \Ext(L,N) = 0$ for some $L\in \rep_\gamma(Q)$ and $N\in \rep_\beta(Q)$.
The rigidity assumption is also necessary here (see Remark \ref{R:acyclicv}).

Since the rigid representations correspond to cluster monomials, these results motivate the following conjecture.
\begin{conjecture}[{Conjecture \ref{conj:Fvertex}}] Let $F=\sum_{\gamma}c_\gamma \b{y}^\gamma$ be the $F$-polynomial of a cluster monomial (of any cluster algebra).
	Then $\gamma$ is a vertex of the Newton polytope of $F$ if and only if $c_\gamma=1$.
\end{conjecture}
\noindent The conjecture is settled for the acyclic quiver case (Corollary \ref{C:Fvertex}). Two other related conjectures are Conjecture \ref{conj:onept} and \ref{conj:vertex}.

We have a complete and general solution to the second part of Problem \ref{p:b}.
By the restriction of a polynomial $F=\sum_{\gamma} c_\gamma \b{y}^\gamma$ to some face ${\sf \Lambda}$ of its Newton polytope,
we mean
$$\sum_{\gamma\mid \gamma\in \sf\Lambda} c_\gamma \b{y}^\gamma.$$
\begin{theorem}[Theorem \ref{T:faces}] \label{T:f} 	Let $\delta$ be the outer normal vector of some facet of the Newton polytope $\N(M)$.	
	Then the restriction of $F_M$ to this facet is given by
	$$ \b{y}^{\dv t_\dtb(M)} \iota_{\delta} \left( F_{ \pi_{\delta} (M) } \right).$$
\end{theorem}
\noindent Here, $t_\dtb$ is the functor constructed in Section \ref{ss:torsion} (see also Theorem \ref{T:tf});
$\pi_\delta(M)$ is a representation of another quiver, which is explained in Section \ref{ss:stablered}, and $\iota_\delta$ is a certain monomial change of variables.
This result can be easily generalized to arbitrary faces of codimension greater than 1 (see Remark \ref{R:anyfaces}).

For a general representation $M$ of an acyclic quiver, we have a complete description of the normal vectors of $\N(M)$. We thus solve the first part of Problem \ref{p:b} in a special case.
Consider the rational polyhedral cones 
$$\cone{\sf \Delta}_0(M)=\{\delta \mid \delta(v)\leq 0, v\in \V(M) \} \text{ and } \cone{\sf \Delta}_1(M)=\{\delta \mid \delta(\alpha-v)\geq 0, v\in {\V}(M) \}.$$  
Here, $\V(M)$ is the set of vertices in $\N(M)$. See \eqref{eq:Delta01} for another description for these cones.
Let ${\sf R}_0(M)$ and ${\sf R}_1(M)$ be the extremal rays of $\cone{\sf \Delta}_0(M)$ and $\cone{\sf \Delta}_1(M)$.
\begin{theorem}[Theorem \ref{T:Newton} and Corollary \ref{C:Newton}] \label{T:intro4} Let $M$ be an $\alpha$-dimensional general representation of an acyclic quiver.
Then the outer normal vectors of $\N(M)$ are precisely ${\sf R}_0(M) \cup {\sf R}_1(M)$, 
and $\N(M)$ can be presented as
$$\{ \gamma\in\mb{R}^{Q_0} \mid \delta_0(\gamma) \leq 0 \text{ for $\delta_0 \in {\sf R}_0(M)$ and } {\delta}_1(\alpha-\gamma) \geq 0 \text{ for ${\delta}_1 \in {\sf R}_1(M)$} \}.$$	
Moreover, if $M$ is rigid, then the rays in ${\sf R}_0(M)$ and ${\sf R}_1(M)$ are either $-{e}_i$ or correspond to real Schur roots.
\end{theorem}
\noindent In this case Theorem \ref{T:f} reduces to a very nice form -- Corollary \ref{C:acyclicf}.
We also show by example that this description does not work in general, even under a rather strong rigidity condition on $M$ (see Remark \ref{R:Newton}).

For Problem \ref{p:c}, following \cite{MTY} we say the support of a polynomial $F=c_\gamma \b{x}^\gamma$ {\em saturated} if $c_\gamma\neq 0$ for any lattice points in the Newton polytope of $F$.
Similarly we say the sub-lattice of a representation $M$ {\em saturated} if for any lattice point $\gamma\in \N(M)$ there is some $\gamma$-dimensional subrepresentation $L\hookrightarrow M$.
It is not clear if the saturation of the support of $F_M$ is equivalent to the saturation of the sub-lattice of $M$. 

We prove the saturation property for a rigid representation of an acyclic quiver, and show by example that the rigidity is necessary (see Remark \ref{R:saturation}).
\begin{theorem}[Theorem \ref{T:saturation}] \label{T:intro5} Let $M$ be a rigid representation of an acyclic quiver. Then both the sub-lattice of $M$ and the support of $F_M$ are saturated.
\end{theorem}
\noindent More generally, we have the following conjecture. The above theorem settles the conjecture in the acyclic quiver case.

\begin{conjecture}[{Conjecture \ref{conj:saturation}}] The support of the $F$-polynomial of a cluster monomial (of any cluster algebra) is saturated.
\end{conjecture}

Finally we discuss the tools for proving these results. 
Besides the theory of tropical $F$-polynomials and their interactions with general presentations \cite{Ft}, the most important ingredient is the construction of the two pairs of functors $(t_\dtb,f_\dtb)$ and $(\tc_\dtb,\fc_\dtb)$ associated to $\delta\in\mb{Z}^{Q_0}$.
The equivalent functors were already considered in \cite{BKT}.

We write any $\delta\in \mb{Z}^{Q_0}$ as $\delta = \delta_+ - \delta_-$ where $\delta_+=\max(\delta,0)$ and $\delta_- = \max(-\delta,0)$.
For $\beta \in \mb{Z}_{\geq 0}^{Q_0}$ we write $P(\beta)$ for $\bigoplus_{i\in Q_0} \beta(i)P_i$.
We denote by $\hom(\delta,M)$ the minimal dimension of $\Hom(C,M)$ where $C$ is the cokernel of 
the projective presentation $P(\delta_-) \to P(\delta_+)$.
Let $\mc{L}(\delta,M)$ be the set of all subrepresentations $L$ of $M$ attaining $\max_{L\hookrightarrow M}{\delta(\dv L)}$.
\begin{theorem}[\cite{BKT}, Theorem \ref{T:maxminsub}] \label{T:tf} The set $\mc{L}(\delta,M)$ contains a unique minimal element $L_{\min}$ and a unique maximal element $L_{\max}$.
	Moreover, $L_1/L_0$ is $\delta$-semistable (in the sense of A. King \cite{Ki}) for any $L_0\subset L_1$ in $\mc{L}(\delta,M)$.
\end{theorem}

\noindent We define the pair of functors $t_\dtb$ and $f_\dtb$ by $t_\dtb(M)=L_{\min}$ and $f_\dtb(M)=M/L_{\min}$,
and the pair of functors $\tc$ and $\fc$ by $\check{t}_\dtb(M)=L_{\max}$ and $\fc_\dtb(M)=M/L_{\max}$. We define the {\em $\delta$-stabilization functor} $\dtb^\perp$ by $\dtb^\perp(M)=L_{\max}/L_{\min}$.
Some properties of these functors are related to the torsion theory (see Theorem \ref{T:torsion}, Proposition \ref{P:relations}, and Corollary \ref{C:univhomo}).

This paper is organized as follows.
In Section \ref{S:Fpoly} we recall the definition of the $F$-polynomial of a representation and its tropical analogue.
In Section \ref{S:Ftrop} we recall from \cite{Ft} how the tropical $F$-polynomials interplay with general presentations.
In Section \ref{S:stab} we construct the functors $(t,f)$, $(\tc,\fc)$, and $\dtb^\perp$ in Theorem \ref{T:maxminsub}.
Then we make connection with the torsion theory in Theorem \ref{T:torsion}.
In Section \ref{S:vertices} we characterize the vertices of the Newton polytopes in Theorem \ref{T:onept} and \ref{T:acyclicv}.
In Section \ref{S:saturation} we prove the saturation property for a rigid representation in Theorem \ref{T:saturation}.
In Section \ref{S:faces} we formulate the restriction of $F$-polynomials to a particular face in full generality in Theorem \ref{T:faces}.
Then we specialize to the case of acyclic quivers in Theorem \ref{T:Newton} and Corollary \ref{C:acyclicf}.

\subsection*{Notation and Conventions}
Throughout we only deal with finite-dimensional basic algebras. So if we write an algebra $A=\mb{C}Q$, we assume implicitly that $Q$ is finite and has no oriented cycles.
For general $A=\mb{C}Q/I$, we allow $Q$ to have oriented cycles.
Although the paper is written in this generality, some of the results are only proved for path algebras.
Sometimes instead of switching between $A=\mb{C}Q/I$ and $A=\mb{C}Q$ we may just say that assume $A$ has no relations.
We denote by $Q_0$ the set of vertices of $Q$.

Unless otherwise stated, unadorned $\Hom$ and other functors are all over the algebra $A$, and the superscript $*$ is the trivial dual for vector spaces.
For direct sum of $n$ copies of $M$, we write $nM$ instead of the traditional $M^{\oplus n}$.
We write $\hom,\ext$ and $\e$ for $\dim\Hom, \dim\Ext$, and $\dim \E$.
When dealing the hereditary algebras, we write $\Ext$ instead of $\Ext^1$.

Any rational ray or normal vector will be represented by an indivisible integral vector, that is, an integral vector with no common divisors.
\begin{align*}
		& \rep(A) && \text{the category of finite-dimensional representations of an algebra $A$} &\\
		& \rep_\alpha(Q) && \text{the space of $\alpha$-dimensional representations of a quiver $Q$} &\\
		& S_i && \text{the simple representation supported on the vertex $i$} &\\
		& P_i && \text{the projective cover of $S_i$} &\\
		& I_i && \text{the injective cover of $S_i$} &\\
		& \dv M && \text{the dimension vector of $M$} & 
\end{align*}

\section{$F$-polynomial of a Representation} \label{S:Fpoly}
Most results in the paper are valid for any base field of characteristic $0$. 
We have a few statements involving the topological Euler characteristics. Instead of switching between $\mb{C}$ and other fields,
we will work with the field of complex numbers from the beginning. 
Since the Euler characteristic plays no role in the proof of the remaining results (e.g., Theorem \ref{T:intro1},\ref{T:intro2},\ref{T:intro4},\ref{T:intro5},\ref{T:tf}), one should understand that there is no essential difficulty to deal with other fields of characteristic $0$.

Let $A$ be a finite-dimensional basic algebra over $\mb{C}$.
Then $A$ can be presented as a quiver algebra modulo an ideal generated by admissible relations: $A=\mb{C}Q/I$.
Throughout we identify the Grothendieck group $K_0(\rep(A))$ with $\mb{Z}^{Q_0}$.
Let $M$ be a finite-dimensional representation of $A$. Following \cite{DWZ2},
\begin{definition} The $F$-polynomial of $M$ is by definition
	$$F_M(\b{y}) = \sum_{\gamma} \chi(\Gr_\gamma(M)) \b{y}^\gamma,$$
	where $\Gr_\gamma(M)$ is the variety parametrizing the $\gamma$-dimensional subrepresentations of $M$, and $\chi(-)$ is the topological Euler characteristic.
\end{definition}

\begin{definition} The {\em tropical $F$-polynomial} $f_M$ of a representation $M$ is the function $(\mb{Z}^{Q_0})^* \to \mb{Z}_{\geq 0}$ defined by
	$$\delta \mapsto \max_{L\hookrightarrow M}{\delta(\dv L)}.$$
The {\em dual} tropical $F$-polynomial $\fc_M$ of a representation $M$ is the function $(\mb{Z}^{Q_0})^* \to \mb{Z}_{\geq 0}$ defined by
	$$\delta \mapsto \max_{M\twoheadrightarrow N}{\delta(\dv N)}.$$
\end{definition}
\noindent Clearly $f_M$ and $\fc_M$ are related by $f_M(\delta)-\fc_M(-\delta)= \delta(\dv M)$.

\begin{definition}
	The {\em Newton polytope ${\N}(M)$ of a representation} $M$ is the convex hull of
	$$\{ \dv L \mid L\hookrightarrow M \}$$
in $\mb{R}^{Q_0}$.	The {\em dual} Newton polytope $\check{\N}(M)$ of a representation $M$ is the convex hull of
	$$\{ \dv N \mid M\twoheadrightarrow N \}$$
in $\mb{R}^{Q_0}$. This is what the authors called Harder-Narasimhan polytope in \cite{BKT}.	
\end{definition}

\begin{remark} We have a few remarks. \begin{enumerate} 
		\item We will see in Theorem \ref{T:onept} that $\chi(\Gr_\gamma(M))=1$ if $\gamma$ is a vertex of $\N(M)$.
		So the Newton polytope of $M$ is the same as the usual Newton polytope of the polynomial $F_M$.
		\item 	The tropical $F$-polynomial $f_M$ is completely determined by the Newton polytope of $M$. In particular, $f_M(\delta) = \max_{\gamma\in\N(M)} \delta(\gamma)$.
		\item	In general $\chi(\Gr_\gamma(M))$ may not be a positive number. If the $F$-polynomial $F_M$ has non-negative coefficients, then the tropical $F$-polynomial $f_M$ is the usual tropicalization of $F_M$.
	\end{enumerate}
\end{remark}

\begin{lemma} \label{L:directsum}  $F_{M\oplus N}= F_MF_N$ for any two representations $M$ and $N$.
Moreover, we have that $f_{M\oplus N} = f_{M} + f_{N}$.
\end{lemma}
\begin{proof} The first statement was proved in \cite[Proposition 3.2]{DWZ2}.
As in its proof, we consider the torus action on $X=\Gr_{\gamma}(M\oplus N)$ induced by the automorphism of $M\oplus N$
$$t\cdot (m,n) = (tm,n)\quad (m\in M,\ n\in N).$$
Since $X$ is complete, we have that $S_0:=\lim_{t\to 0} tS \in X$ for any $S\in X$. 
In fact, it is easy to check that $S_0 \in X^{\mb{C}^*}$.
But any $S_0$ is $\mb{C}$-fixed if and only if $S_0$ splits into $S_0=L\oplus L'$ for some $L\subset M$ and $L'\subset N$.
	Conversely, for any $L\subset M$ and $L'\subset N$ we have that $L\oplus L' \subset M\oplus N$.
\end{proof}

\section{Tropical $F$-polynomials and General Presentations} \label{S:Ftrop}
\subsection{Tropical $F$-polynomials} \label{ss:Ftrop}
We keep assuming that $A=\mb{C}Q/I$.
We denote by $P_i$ the indecomposable projective representation of $A$ corresponding the vertex $i$ of $Q$.
For $\beta \in \mb{Z}_{\geq 0}^{Q_0}$ we write $P(\beta)$ for $\bigoplus_{i\in Q_0} \beta(i)P_i$.
Following \cite{DF} we call a homomorphism between two projective representations, a {\em projective presentation} (or presentation in short).

\begin{definition}\footnote{The $\delta$-vector is the same one defined in \cite{DF}, but is the negative of the $\g$-vector defined in \cite{DWZ2}. }
	 The {\em $\delta$-vector} (or {\em reduced weight vector}) of a presentation 
	 $$d: P(\beta_-)\to P(\beta_+)$$
	is the difference $\beta_+-\beta_- \in \mb{Z}^{Q_0}$.
	When working with injective presentations 
	$$\dc: I(\check{\beta}_+)\to I(\check{\beta}_-),$$
	 we call the vector $\check{\beta}_+ - \check{\beta}_-$ the {\em $\check{\delta}$-vector} of $\dc$.
\end{definition}
\noindent Let $\nu$ be the Nakayama functor $\Hom(-,A)^*$.
There is a map still denoted by $\nu$ sending a projective presentation to an injective one
$$P_-\to P_+\ \mapsto\ \nu(P_-) \to \nu(P_+).$$
Note that if there is no direct summand of form $P_i\to 0$, then $\Ker(\nu d) = \tau\Coker(d)$ where $\tau$ is the classical Auslander-Reiten translation \cite{ASS}.

\begin{definition}[{\cite{DWZ2,DF}}] Given any projective presentation $d: P_-\to P_+$ and any $N\in \rep(A)$, we define $\Hom(d,N)$ and $\E(d,N)$ to be the kernel and cokernel of the induced map:
	\begin{equation} \label{eq:longexact} 0\to \Hom(d,N)\to \Hom(P_+,N) \xrightarrow{} \Hom(P_-,N) \to \E(d, N)\to 0.
	\end{equation}
Similarly for an injective presentation $\dc: I_+\to I_-$, we define $\Hom(M,\dc)$ and $\Ec(M,\dc)$ to be the kernel and cokernel of the induced map $\Hom(M,I_+) \xrightarrow{} \Hom(M,I_-)$.
It is clear that 
$$\Hom(d,N) = \Hom(\Coker(d),N)\ \text{ and }\ \Hom(M,\dc) = \Hom(M,\Ker(\dc)).$$

For any representation $M$, we denote by $d_M$ (resp. $\dc_M$) its minimal projective (resp. injective) presentation, and by $\delta_M$ (resp. $\check{\delta}_M$) the weight vector of $d_M$ (resp. $\dc_M$).
Given any two representations $M$ and $N$, we define 
$$\E(M,N):= \E(d_M,N)\ \text{ and }\ \Ec(M,N):= \Ec(M,\dc_N).$$

A presentation $d$ is called {\em rigid} if $\E(d,\Coker(d))=0$. Such a presentation has a dense $\Aut(P_-)\times \Aut(P_+)$-orbit in $\Hom(P_-,P_+)$ \cite{DF}.
A representation $M$ is called {\em $\E$-rigid} if $\E(M,M)=0$.
This is equivalent to say that the minimal presentation of $M$ is rigid.
According to Lemma \ref{L:E}.(3) below, this is what the authors called $\tau$-rigid in \cite{AIR}.
Similarly, a representation $M$ is called $\Ec$-rigid if $\Ec(M,M)=0$.
It is called $\E$-birigid if it is both $\E$-rigid and $\Ec$-rigid.
For Jacobian algebras, $\E$-rigid is equivalent to $\Ec$-rigid \cite{DWZ2}.
These representations are an important class of representations in connection with the cluster algebra theory.
\end{definition}

\begin{lemma}[{\cite[Section 5, Lemmas 3.4 and 7.5]{DF}}] \label{L:E} We have the following properties \begin{enumerate}
\item	Any exact sequence $0\to L \to M \to N \to 0$ in $\rep(A)$ gives the long exact sequence:
	$$0\to \Hom(d,L)\to \Hom(d,M) \to \Hom(d,N) \to \E(d, L) \to \E(d, M) \to \E(d, N)\to 0.$$
\item	$\E(M,N) \supseteq \Ext^1(M,N)$ for any $M$ and $N$.
\item	$\E(d,N) \cong \Hom(N, \nu d)^*$ for any $d$ and $N$.
\end{enumerate}	
\end{lemma}
\noindent Readers can easily formulate the analogous statements for $\Ec$.

Any $\delta\in \mb{Z}^{Q_0}$ can be written as $\delta = \delta_+ - \delta_-$ where $\delta_+=\max(\delta,0)$ and $\delta_- = \max(-\delta,0)$. Here the maximum is taken coordinate-wise.
We put $\PHom(\delta):=\Hom(P(\delta_-),P(\delta_+))$.
We say that a {\em general} presentation in $\PHom(\delta)$ has property $A$ if there is some open (and thus dense) subset $U$ of $\PHom(\delta)$ such that all presentations in $U$ have property $A$.
For example, a general presentation $d$ in $\PHom(\delta)$ has the following properties:
\begin{enumerate}
   \item $\Hom(d,N)$ has constant dimension for a fixed $N\in \rep(A)$.
   \item $\Gr_\gamma(\Coker(d))$ has constant topological Euler characteristic.
\end{enumerate}
Note that (1) implies that $\E(d,N)$ has constant dimension on $U$ as well. 
If we apply (1) to $N=A^*$, then $\Coker(d)$ has a constant dimension vector.
It follows from (2) that $U$ can be chosen such that $\Coker(d)$ has a constant $F$-polynomial.
We denote by $\Coker(\delta)$ the cokernel of a general presentation in $\PHom(\delta)$.
Note that from the discussion above, any rigid presentation is general.
\begin{definition} We denote by $\hom(\delta,N)$ and $\e(\delta,N)$ the value of 
	$\hom(d,N)$ and $\e(d,N)$ for a general presentation $d\in \PHom(\delta)$.
$\hom(M,\dtc)$ and $\ec(M,\dtc)$ are defined analogously.
\end{definition}
\noindent If $\delta = \delta_1 + \delta_2$, then it is easy to show by the upper semi-continuity \cite[Lemma 2.11]{Ft} that $\hom(\delta,M) \leq \hom(\delta_1,M) + \hom(\delta_2,M)$ for any $M$.

When paired with a dimension vector or evaluated by some $f_M$, 
a weight $\delta$ is viewed as an element in $(\mb{Z}^{Q_0})^*$ via the usual dot product.
It follows from \eqref{eq:longexact} that for any presentation $d$ of weight $\delta$,
\begin{align*} \delta(\dv M) &= \hom(d,M) - \e(d,M);\\
\check{\delta} (\dv M) &= \hom(M,\dc) - \ec(M,\dc).
\end{align*}

The following is one of the main results in \cite{Ft}.
\begin{theorem}[{\cite[Theorem 2.6]{Ft}}] \label{T:HomE} Let $M$ be any representation of $A$.
	For any $\delta,\dtc\in \mb{Z}^{Q_0}$, there are some $n,\check{n}\in\mb{N}$ respectively such that
	\begin{align*}
	{f}_M(n\delta) &= \hom(n\delta,M), & \fc_M(-n\delta) &= {\e}(n\delta,M);\\
 \fc_M(\check{n}\check{\delta}) &= \hom(M,\check{n}\check{\delta}), & {f}_M(-\check{n}\check{\delta}) &= \ec(M,\check{n}\check{\delta}).\end{align*}
Moreover, the equality still holds if $n$ is replaced by $kn$ for any $k\in\mb{N}$.
\end{theorem} 

\begin{remark} \label{r:n=1} In general, we cannot guarantee $n=1$ unless in the following special situations: \begin{enumerate}
		\item $\delta$ satisfies $\e(\delta,\delta):=\min\{\e(d_1,d_2)\mid d_1,d_2\in \PHom(\delta) \}=0$ (see \cite{Ft}).
		\item $A$ is the path algebra of an acyclic quiver and $M$ is a general representation (\cite{S2}).
	\end{enumerate}
For more general statements in this direction, see \cite[Theorem 2.22 and Corollary 2.23]{Ft}.
\end{remark}

Consider the sets 
\begin{align*}
{\sf \Delta}_0(M)&=\{\delta\in \mb{Z}^{Q_0} \mid \hom(n\delta,M)=0\ \text{ for some $n\in\mb{N}$} \}, \\
{\sf \Delta}_1(M)&=\{\delta\in \mb{Z}^{Q_0} \mid {\e}(n\delta,M)=0\ \text{ for some $n\in\mb{N}$} \}.
\end{align*}
Let ${\sf V}(M)$ and $\check{\sf V}(M)$ be the set of vertices in $\N(M)$ and $\check{\N}(M)$.
\begin{corollary}[{\cite[Corollary 4.11]{Ft}}] \label{C:HE0} ${\sf \Delta}_0(M)$ $($resp. ${\sf \Delta}_1(M)$$)$ are precisely the lattice points in the polyhedral cone defined by $\delta(v) \leq 0$ for all $v \in {\V}(M)$ $($resp. $\delta(v) \geq 0$ for all $v \in {\sf \check{V}}(M)$$)$.
\end{corollary}
\begin{remark} \label{r:HE0}
For $A=\mb{C}Q$, a general $\alpha$-dimensional representation is the cokernel of a general presentation of weight $\delta_\alpha = \innerprod{\alpha, -}$, where $\innerprod{-,-}$ is the Euler form of the quiver.
So all dimension vectors $\alpha$ such that $\hom(\alpha, \beta)=0$ $($resp. $\ext(\alpha, \beta)=0$$)$ are precisely the lattice points in the polyhedral cone defined by $\alpha\in \mb{Z}_{\geq 0}^{Q_0}$ and $\innerprod{\alpha, \gamma} \leq 0$ for all $\gamma \in {\V}(M)$ (resp. $\innerprod{\alpha, \gamma} \geq 0$ for all $\gamma \in {\sf \check{V}}(M)$).	
Here, $M$ is a general representation of dimension $\beta$, and \begin{align*}
\hom(\gamma,\beta) &= \min\{\hom(M,N) | M \in \rep_\gamma(Q), N \in \rep_\beta(Q)\},\\
\ext(\gamma,\beta) &= \min\{\ext(M,N) | M \in \rep_\gamma(Q), N \in \rep_\beta(Q)\}.
\end{align*}
\end{remark}

\subsection{Newton Polytopes of $F$-polynomials}
Let $V$ be a $\mathbb R$-vector space.
To a non-empty compact convex subset ${\sf P}$ of $V$, we associate its
support function $\psi_{\sf P}:V^*\to\mathbb R$, which maps a linear function
$f\in V^*$ to the maximal value $f$ takes on ${\sf P}$.
Then $\psi_{\sf P}$ is a sublinear function on $V^*$. 
As explained in \cite[Section 4.2]{BK}, one can recover ${\sf P}$ from the datum of $\psi_{\sf P}$ by the Hahn-Banach theorem
$${\sf P}=\{v\in V\mid \alpha(v)\leq \psi_{\sf P}(\alpha),\ \forall\alpha\in V^*\},$$
and the map ${\sf P}\mapsto\psi_{\sf P}$ is a bijection from the set of all
non-empty compact convex subsets of $V$ onto the set of all
sublinear functions on $V^*$. 
In our setting of ${\sf P}=\N(M)$, the support function is given by $\delta \mapsto f_M(\delta)$.
It follows easily from Theorem \ref{T:HomE} that

\begin{corollary}[{\cite[Theorem 4.1]{Ft}}] \label{C:allfaces} The Newton polytope $\N(M)$ is defined by 
	$$\{\gamma\in \mb{R}^{Q_0}\mid \delta(\gamma)\leq \hom(\delta,M),\ \forall \delta \in \mb{Z}^{Q_0} \}.$$
	The dual Newton polytope $\check\N(M)$ is defined by 
	$$\{\gamma\in \mb{R}^{Q_0}\mid \check{\delta}(\gamma)\leq \hom(M,\check{\delta}),\ \forall \check{\delta} \in \mb{Z}^{Q_0} \}.$$
\end{corollary}

\noindent We know a priori that the Newton polytope has a (finite) hyperplane representation. 
In fact we only need those $\delta$-vectors which are outer normal vectors of $\N(M)$.
It is an interesting problem to find a finite set of $\delta$-vectors determining the Newton polytope.
This is achieved for general representations of any acyclic quiver (Theorem \ref{T:Newton}).

We remark that for the presentation in Corollary \ref{C:allfaces} we cannot further require that $\delta$ is indivisible.
An $\delta$-vector is called {\em indecomposable} (\cite{DF}) if a general presentation in $\PHom(\delta)$ is indecomposable.
\begin{corollary}[{\cite[Corollary 4.3]{Ft}}] \label{C:indnormal} An indivisible outer normal vector of $\N(M)$ is an indecomposable $\delta$-vector.
\end{corollary}

\subsection{General Presentations and the $\tau$-tilting theory}
For readers more familiar with $\tau$-tilting theory \cite{AIR}, we include this short section to explain the relationship between general presentations and $\tau$-tilting theory.
Recall that a decorated representation of $A=\mb{C}Q/I$ is a pair $\mc{M}=(M, V)$, where $M \in \rep(A)$ and $V\in\rep(\mb{C}Q_0)$ is called the decorated part of $\mc{M}$. We call a presentation {\em negative} if it is of the form  $P\to 0$, which will be denoted by $P[1]$. 

As explained at the beginning of \cite[Section 7]{DF}, there is a bijection between the category of {\em decorated representations} and the category of presentations. 
Note that a decorated representation $(M,V)$ becomes an ordinary representation if we set the decorated part $V$ to be zero.
Specifically, this bijection maps any representation $(M,0)$ to its minimal presentation $d_M$, and simple negative representation $(0,S_i)$ to $P_i[1]$. Conversely, any presentation $f$ splits into $f' \oplus P_{-}[1]$ where $f'$ has no summands of negative presentations.
We map $f$ to the decorated representation $({\rm Coker}(f), P_-/{\rm rad}(P_-))$.
We observe that if the summands of a rigid $f$ contain some $P_i[1]$, then the cokernel $M$ of $f$ cannot be supported on $i\in Q_0$.
This is because $M(i)\cong {\rm Hom}_A(P_i,M)=0$.
In this way, support $\tau$-tilting pairs \cite{AIR} can be viewed as some special kind of decorated representations.

Under this bijection a $\tau$-rigid module corresponds to a $\E$-rigid module by \cite[Lemma 7.5]{DF},
a $\tau$-tilting module corresponds to a {$\E$-rigid module} with $|Q_0|$ indecomposable summands (they turn out be maximal rigid by \cite[Theorem 5.4]{DF}), and a support $\tau$-tilting pair $(M,P)$ corresponds to a {maximal rigid presentation} $d_M\oplus P[1]$.

The framework of \cite{DF,Ft} and the current paper goes beyond $\tau$-tilting theory because the presentations are not assumed to be rigid. We refer the readers to the article \cite{GPtau}, which explains in more details how the results in \cite[Section 5]{DF} imply the main results in \cite{AIR}.

\section{Functors associated to $\delta$} \label{S:stab}
\subsection{The Construction} A. King introduced Mumford's GIT into the setting of quiver representation theory \cite{Ki}.
He derived a nice criterion for the (semi)-stability of a representation.
Here we state his criterion as our definition for the stability.
\begin{definition}[{\cite[Proposition 3.1]{Ki}}]  \label{D:King} A representation $M$ is called {\em $\delta$-semistable} (resp. {\em $\delta$-stable}) if $\delta(\dv M)=0$ and $\delta(\dv L)\leq 0$ $($resp. $\delta(\dv L)<0$$)$ for any non-trivial subrepresentation $L$ of $M$.
\end{definition}
\noindent The set $\ss{\delta}{A}$ of all $\delta$-semistable representations form an abelian subcategory of $\rep(A)$.
A special case of Theorem \ref{T:HomE} is the following.
\begin{corollary} \label{C:ss} A representation $M$ is $\delta$-semistable iff $\hom(n\delta,M) = \delta(\dv M)=0$ for some $n\in\mb{N}$.
\end{corollary}

Now we start our construction.
\begin{lemma}\label{L:sub=h} Let $L$ be any subrepresentation of $M$. Then $\delta(\dv L)= \hom(\delta,M)$ if and only if $\hom(\delta,M/L)=\e(\delta,L)=0$.
	Moreover, if $L'$ is another such subrepresentation, that is,  $\delta(\dv L')=\hom(\delta,M)$, then both 
	$L\cap L'$ and $L+ L'$ are such subrepresentations.
\end{lemma}	
\begin{proof} Assume that $\delta(\dv L)= \hom(\delta,M)=h$. We have that
	$h=\delta(\dv L)\leq \hom(\delta,L)\leq \hom(\delta,M)=h$, so $\hom(\delta,L)=h$ and $\e(\delta,L)=0$.	
	Similarly, $\e(\delta,N)=\e(\delta,M)$ and $\hom(\delta,N)=0$ where $N=M/L$.
	Conversely, if $\hom(\delta,N)=\e(\delta,L)=0$, then there is an exact sequence for a general presentation $d$ of weight $\delta$
	$$0\to \Hom(d,L) \to \Hom(d,M)\to \Hom(d,M/L)=0.$$ 
    This implies that 
	$\delta(\dv L)= \hom(\delta,L)-0 =\hom(\delta,M).$
	
	Now suppose that there is another such representation, say $L'$. 
	Let $I=L\cap L'$ and $C$ be the cokernel of $I \hookrightarrow L$,
	which is also isomorphic to the cokernel of $L' \hookrightarrow L+ L'$.
	Since $\hom(\delta, M/L')=0$, we have that $\hom(\delta, (L+ L')/L')=0$ so $\hom(\delta, C)=0$.
	We conclude that $\hom(\delta,I)=h$ as well.
	In the meanwhile, $\e(\delta, L)=0$.
	We read from the part of the long exact sequence of Lemma \ref{L:E}.(1) for a general presentation $d$ of weight $\delta$
	$$0=\Hom(d,C) \to \E(d,I) \to \E(d,L) =0$$
	that $\e(\delta,I)=0$. Hence,  $\delta(\dv I)=h$.
	Finally, $\delta(\dv(L+L')) = h$ follows from the isomorphism $(L+L')/L'\cong L/(L\cap L')$.
\end{proof}

Let $\mc{L}(\delta,M)$ be the set of all subrepresentations $L$ of $M$ such that $\delta(\dv L) = f_M(\delta)$.
Since $f_M$ is piecewise linear, we also have that
\begin{equation} \label{eq:f=hom} \mc{L}(\delta,M) = \{L\subseteq M \mid  n\delta(\dv L)=\hom(n\delta,M) \}, \end{equation}
where $n$ is any number in Theorem \ref{T:HomE}. 
Since we always have $f_M(n\delta)\leq \hom(n\delta,M)$,
$$\mc{L}(\delta,M) = \{L\subseteq M \mid  n\delta(\dv L)=\hom(n\delta,M) \text{ for some $n\in\mb{N}$}  \}.$$
The following theorem was proved in \cite[1.4]{BKT}. Here we give another proof.

\begin{theorem}[{\cite{BKT}}] \label{T:maxminsub} The set $\mc{L}(\delta,M)$ contains a unique minimal element $L_{\min}$ and a unique maximal element $L_{\max}$.
	Moreover, $L_1/L_0$ is $\delta$-semistable for any $L_0\subset L_1$ in $\mc{L}(\delta,M)$.
\end{theorem}

\begin{proof} Consider the equivalent definition of $\mc{L}(\delta,M)$ \eqref{eq:f=hom}, and fix such an $n$.
	By Lemma \ref{L:sub=h} for weight equal to $n\delta$, the set is closed under intersection and summation, so it must contain a unique minimal element and a unique maximal element.
	For the last statement, we have that $\delta(\dv(L_1/L_{0}))=0$.
	By Lemma \ref{L:sub=h} $\hom(n\delta,M/L_{0})=0$ so $\hom(n\delta,L_1/L_{0})=0$.
	Hence $L_1/L_{0}$ is $\delta$-semistable by Corollary \ref{C:ss}.
\end{proof}
\noindent We remark that it is possible to have $L_{\min} = L_{\max}$.

\begin{definition} \label{D:stabilization} 
	We define the {\em stabilization functor} $\dtb^\perp$ associated to $\delta\in \mb{Z}^{Q_0}$ by 
	$$\dtb^\perp(M)=L_{\max}/L_{\min},$$
	 where $L_{\min}$ and $L_{\max}$ are the subrepresentations in Theorem \ref{T:maxminsub}.
\end{definition}
\noindent We shall see in the next subsection that $L_{\min}$ and $L_{\max}$ are defined by the functors associated to some torsion theory, so $\dtb^\perp$ is in fact functorial.

\begin{remark} \label{R:dual} We can easily use $\Hom(-,\check{\delta})$ and $\Ec(-,\check{\delta})$ to reformulate the results in this subsection.
	Instead of working with injective presentations, one may consider the relations $\Hom(\delta,M)= \Ec(M,-\delta)$ and $\E(\delta,M)= \Hom(M,-\delta)$.	
\end{remark}

\subsection{Relation to the Torsion Theory} \label{ss:torsion}
Recall that a {\em torsion class} in $\rep A$ is a full subcategory of $\rep A$ which is closed under images, direct sums, and extensions, 
and a {\em torsion-free class} in $\rep A$ is a full subcategory of $\rep A$ which is closed under subrepresentations, direct sums, and extensions.
Consider the torsion free class
$${\mc{F}}(\delta) = \{N\in \rep(A)\mid \hom(\delta,N) =0 \}$$
and the torsion class
$$\check{\mc{T}}(\delta) = \{L\in \rep(A) \mid \e(\delta, L) =0 \}.$$
We denote associated pair of functors by $(t_\delta,f_\delta)$ and $(\tc_\delta,\fc_\delta)$ (see \cite[Proposition VI.1.4]{ASS}).
To be more explicit, for any representation $M$,
$t_\delta(M)$ is the smallest subrepresentation $L$ of $M$ such that $\hom(\delta,M/L)=0$, and
$\tc_\delta(M)$ is the largest subrepresentation $L$ of $M$ such that $\e(\delta,L)=0$ (see \cite[the proof of Proposition VI.1.4]{ASS}).

\begin{lemma} \label{L:L=tf} If $f_M(n\delta)=\hom(n\delta,M)$, then $L_{\max}=\tc_{n\delta}(M)$ and $L_{\min}=t_{n\delta}(M)$.
\end{lemma}
\begin{proof} Let $L$ be any subrepresentation of $M$ containing $L_{\max}$ such that $\e(n\delta,L)=0$.
	Consider the exact sequence
$$0= \Hom(n\delta, M/L_{\max}) \to \Hom(n\delta, M/L) \to \E(n\delta, L/L_{\max})=0,$$
where the vanishing of $\Hom(n\delta, M/L_{\max})$ and $\E(n\delta, L/L_{\max})$ are due to Lemma \ref{L:sub=h} and Lemma \ref{L:E}.(1) respectively.
Hence $\hom(n\delta, M/L)=0$.
So $n\delta (\dv L) = \hom(n\delta,M)$ by Lemma \ref{L:sub=h}.
By the maximality of $L_{\max}$, we have $L=L_{\max}$.
The other half can be proved similarly.
\end{proof}


Let us consider the two sets
\begin{align*} {\mc{F}}(\br{\delta}) &= \{N\in \rep(A)\mid \hom(n\delta,N) =0 \text{ for some } n\in\mb{N} \},\\
\check{\mc{T}}(\br{\delta}) &= \{L\in \rep(A) \mid \e(n\delta, L) =0   \text{ for some } n\in\mb{N}  \}.
\end{align*}
\noindent Clearly we have that ${\mc{F}}(m\delta) \subseteq {\mc{F}}(\br{\delta})$ and $\check{\mc{T}}(m\delta) \subseteq \check{\mc{T}}(\br{\delta})$ for any $m\in \mb{N}$.

\begin{lemma} \label{L:} ${\mc{F}}(\br{\delta})$ is a torsion-free class, and $\check{\mc{T}}(\br{\delta})$ is a torsion class. 
\end{lemma}
\begin{proof} We only verify that $\check{\mc{T}}(\br{\delta})$ is closed under extension because closed under direct sum and image is rather trivial. Let $0\to L \to M\to N\to 0$ be an exact sequence with $L,N\in \check{\mc{T}}(\br{\delta})$.
	Suppose that $\e(n_1\delta, L) =0$ and $\e(n_2\delta, N) =0$, then $\e(n_1n_2\delta, L) =\e(n_1n_2\delta, N) =0$ by the upper semi-continuity.
	Hence $\e(n_1n_2\delta, M)=0$.
\end{proof}
In view of Theorem \ref{T:HomE}, the torsion pairs $(\mc{T}(\br{\delta}),\mc{F}(\br{\delta}))$ and $(\check{\mc{T}}(\br{\delta}),\check{\mc{F}}(\br{\delta}))$ are equivalent to the ones $(\mc{T}_{\delta},\br{\mc{F}}_{\delta})$ and $(\br{\mc{T}}_\delta,\mc{F}_{\delta})$ introduced in \cite{BKT}, and further studied in \cite{BST} and \cite{As}.
Let us recall their definitions.
\begin{align*} \mc{T}_{\delta} : &= \{M\in \rep(A) \mid \delta(\dv N)>0 \text{ for any quotient representation $N\neq 0$ of $M$} \}  \\
\br{\mc{F}}_{\delta} : &= \{M\in \rep(A) \mid \delta(\dv L)\leq 0 \text{ for any subrepresentation $L$ of $M$}\};
\end{align*}
\begin{align*} \br{\mc{T}}_\delta : &= \{M\in \rep(A) \mid \delta(\dv N)\geq 0 \text{ for any quotient representation $N$ of $M$} \}  \\
\mc{F}_{\delta} : &= \{M\in \rep(A) \mid \delta(\dv L)<0 \text{ for any subrepresentation $L\neq 0$ of $M$} \}.
\end{align*}
Since we do not have a homological interpretation for $\mc{T}(\br{\delta})$ and $\check{\mc{F}}(\br{\delta})$,
the definition provides us with a concrete description for them from another viewpoint.

\begin{definition} We denote by $(t_{\br{\delta}},f_{\br{\delta}})$ and $(\tc_{\br{\delta}},\fc_{\br{\delta}})$ the pairs of functors associated to the torsion pair $({\mc{T}}(\br{\delta}),{\mc{F}}(\br{\delta}))$ and $(\check{\mc{T}}(\br{\delta}),\check{\mc{F}}(\br{\delta}))$.
\end{definition}
\noindent Similarly, $t_\dtb(M)$ is the smallest subrepresentation $L$ of $M$ such that $\hom(n\delta,M/L)=0$ for some $n\in\mb{N}$, and $\tc_{\br{\delta}}(M)$ is the largest subrepresentation $L$ of $M$ such that $\e(n\delta,L)=0$ for some $n\in\mb{N}$.

\begin{theorem} \label{T:torsion} We have that for any representation $M$ and any $\delta \in \mb{Z}^{Q_0}$,
\begin{align*}t_{\dtb}(M)=L_{\min}\ &\text{ and }\  f_{\dtb}(M)=M/L_{\min};\\
 \tc_{\dtb}(M)=L_{\max}\ &\text{ and }\ \fc_{\dtb}(M)=M/L_{\max}.\end{align*}
	In particular, we have for any $L\in\mc{L}(\delta,M)$ that
	$$\Hom(t_\dtb(M),M/L) = 0\ \text{ and }\ \Hom(L,\fc_\dtb(M)) = 0.$$
\end{theorem}

\begin{proof} We have seen in Lemma \ref{L:L=tf} that $L_{\max}$ is the largest subrepresentation $L$ of $M$ such that $\e(n\delta,L)=0$ where $n$ is any number such that $f_M(n\delta)=\hom(n\delta,M)$.
	Suppose that $\e(m\delta,L)=0$ for some $m\in\mb{N}$, then $\e(mn\delta,L)=0$ by the upper semi-continuity. By the moreover part of Theorem \ref{T:HomE} we have that $f_M(mn\delta)=\hom(mn\delta,M)$ so $L\subseteq L_{\max}$. 
	Hence $\tc_{\dtb}(M)=L_{\max}$ by the above comments. 
		
	By Lemma \ref{L:sub=h} $\mc{L}(\delta,M) \subseteq \check{\mc{T}}(\dtb)$. So $\Hom(L,\fc_\dtb(M)) = 0$ for all $L\in\mc{L}(\delta,M)$.
	The argument for $t_{\dtb}(M)$ and $\Hom(t_\dtb(M),M/L) = 0$ is similar.
\end{proof}

\begin{remark} Theorem \ref{T:HomE} and Lemma \ref{L:L=tf} imply that $t_{\dtb}(M) = t_{n\delta}(M)$ and $\tc_{\dtb}(M) = \tc_{n\delta}(M)$ for some $n\in\mb{N}$ but in general we only have that $t_{\dtb}(M) \subseteq t_{m\delta}(M)$ and $\tc_{\dtb}(M) \supseteq \tc_{m\delta}(M)$ for any $m\in \mb{N}$. But if $\e(\delta,\delta)=0$, we have that $t_{\dtb}(M) = t_{\delta}(M)$ and $\tc_{\dtb}(M) = \tc_{\delta}(M)$ by Remark \ref{r:n=1}.(1).
\end{remark}

In the following proposition we ignore all the subscripts $\dtb$ for the functors $t,f,\tc$, and $\fc$.
\begin{proposition} \label{P:relations} $t$ and $\check{t}$ are subfunctors of the identity functor satisfying $tt=t$ and $tf=0$.
	The functors $f$ and $\check{t}$ commute and $\check{t} f=f \check{t} = \dtb^\perp$.
	Moreover, we have the following relations
	$\check{t}t=t=t\check{t}, \fc f=\fc=f\fc$, and $t\fc=0=\fc t$.	
\end{proposition}

\begin{proof} The first statement follows from \cite[Proposition VI.1.4]{ASS}.	
	We claim that $\check{t} f(M)=f \check{t} (M) = L_{\max}/L_{\min}$. 
	To show $f \check{t} (M) = L_{\max}/L_{\min}$, we need to verify that $L_{\min}$ is the minimal subrepresentation $L$ of $L_{\max}$ such that $\delta(\dv L)= f_{L_{\max}}(\delta)$. 
	Since any subrepresentation of $L_{\max}$ is also a subrepresentation of $M$,
	we have that $f_{L_{\max}}(\delta)=f_M(\delta)$ and $L_{\min}$ is the required minimal subrepresentation of $L_{\max}$.
	The dual argument shows that $\check{t} f(M)= L_{\max}/L_{\min}$.
	The other relations are rather easy to verify as well.
\end{proof}

\noindent Hence those functors fit into the following diagram of exact sequences
$$\crossminmax{t(M)}{t(M)}{0}{\check{t}(M)}{M}{\fc(M)}{\dtb^\perp(M)}{f(M)}{\fc(M)}$$
\noindent In particular, we see that $M$ is filtered by three factors $t_{\dtb}(M),\fc_{\dtb}(M)$, and $\dtb^\perp(M)$.
Moreover there is no homomorphism from $t_{\dtb}(M)$ to $\dtb^\perp(M)$ and from $\dtb^\perp(M)$ to $\fc_{\dtb}(M)$.

Suppose that $\hom(M,N)=h$. We choose a basis of $\Hom(M,N)$ and take $hM \to N$ to be the canonical map with respect to this basis.
We call this map a {\em universal homomorphism} from $\op{add}(M)$ to $N$. It is unique up to the action of $\Aut(hM)$.

\begin{corollary} \label{C:univhomo} Suppose that $d$ is a rigid presentation with weight $\delta$ and cokernel $C$.
	Then $t_\delta(M)$ is the image of the universal homomorphism $hC \to M$ while $\tc_\delta(M)$ is the kernel of the universal homomorphism $M \to e \Ker(\nu d)$,
	where $h=\hom(\delta,M)$ and $e=\e(\delta,M)$.
\end{corollary}
\begin{proof} Let $I$ be the image of the universal homomorphism $hC\to M$.
	Then by \cite[Lemma VI.1.9]{ASS} and Lemma \ref{L:E}.(2), $\mc{T}(\delta)=\op{Gen}(C)$ and thus $I \in \mc{T}(\delta)$.
	We obtain a sequence 
	$$\Hom(C,hC)\twoheadrightarrow \Hom(C,M)\to \Hom(C,M/I)\to \E(C,I)= 0,$$
	which is exact at $\Hom(C,M/I)$.
	The last three terms are a part of the long exact sequence in Lemma \ref{L:E}.(1).
	Since the composition $hC\to M\to M/I$ is zero, $\Hom(C,M/I)$ has to be zero.
	We conclude that $\Hom(\delta, M/I)=0$, i.e., $M/I \in \mc{F}(\delta)$.
	Hence $I$ must be $t_\delta(M)$.
	The other half can be proved by the dual argument.
\end{proof}

\begin{remark} The rigidity is necessary in the above corollary. 
	Consider the $3$-arrow Kronecker quiver $\Kronthree{1}{2}$.
	Let $C$ and $M$ be general representations of dimension $(1,2)$ and $(2,1)$ respectively.
	Note that $C=\Coker(1,-1)$ and it is not rigid since $(1,-1)\cdot (1,2) = -1<0$.
	The dimension vectors of subrepresentations of $M$ are $(0,0),(0,1),(1,1)$, and $(2,1)$.
	So by Theorem \ref{T:HomE} $\hom(C,M)=1$.
	It is easy to check that the image of a nonzero homomorphism $C\to M$ has dimension $(1,1)$,
	but $\dv t_{\delta}(M) = \dv t_{\dtb}(M) = (2,1)$.
\end{remark}

\section{Vertices of ${\N}(M)$} \label{S:vertices}
Recall that $\V(M)$ is the set of vertices of $\N(M)$. 
\begin{lemma} \label{L:vertex} $\gamma\in {\V}(M)$ if and only if it is the dimension vector of $t_\dtb(M)$ or $\check{t}_\dtb(M)$ for some weight $\delta \in \mb{Z}^{Q_0}$.
\end{lemma}
\begin{proof} Let $\gamma$ be a vertex of $\N(M)$. We choose a weight $\delta\in\mb{Z}^{Q_0}$ such that $\delta(-)$ reaches the maximum {\em only} at $\gamma$.
	In fact, any generic point $\delta$ in the normal cone attached to $\gamma$ will do the job.
	Then $t_\dtb(M) = \check{t}_\dtb(M)$ and $\gamma=\dv t_\dtb(M)$.
	
	Conversely, let $L=t_\dtb(M)$ or $\check{t}_\dtb(M)$ for some $\delta$. Recall the set 	$\mc{L}(\delta,M)$ consisting of
all subrepresentations $L$ of $M$ such that $\delta(\dv L) = f_M(\delta)$.
	Due to Corollary \ref{C:allfaces}, the convex hull of the dimension vectors of elements in $\mc{L}(\delta,M)$ is a face of $\N(M)$, and $\dv L$ is the minimal or maximal dimension on this face.
	So $\dv L$ must be a vertex of this face, and thus a vertex of $\N(M)$.
\end{proof}

\begin{theorem} \label{T:onept}  If $\gamma \in {\V}(M)$ then $\Gr_\gamma(M)$ must be a point.
\end{theorem}
\begin{proof} Given $\gamma\in \V(M)$, we knew from Lemma \ref{L:vertex} that there is some weight $\delta$ such that 
	the subrepresentation $t_\dtb(M)$ or $\tc_\dtb(M)$ has dimension $\gamma$.
	By Theorem \ref{T:maxminsub}, such a subrepresentation is unique.
	
\end{proof}

\begin{definition} For any $\gamma\in \V(M)$, we call the unique subrepresentation $L$ with $\dv L = \gamma$ a {\em vertex subrepresentation} of $M$.
	The {\em vertex quotient} representation is defined analogously.
\end{definition}

\begin{problem} For general algebras $A$ what condition should be imposed on $M$ such that
	$\Gr_\gamma(M)$ is a point if and only if $\gamma\in {\V}(M)$?
\end{problem}

\begin{conjecture} \label{conj:onept} For finite-dimensional Jacobian algebras and $M=\Coker(\delta)$,
	$\Gr_\gamma(M)$ is a point if and only if $\gamma\in {\V}(M)$.
\end{conjecture}

\begin{remark} \label{R:onept} The generic assumption is necessary here.
	Consider the $3$-arrow Kronecker quiver $\Kronthree{1}{2}$. Let $L,N$ be two general representations of dimension $(1,1)$.
	Then $\ext(N,L)=1$ so a general representation of dimension $(2,2)$ does not have a subrepresentation of dimension $(1,1)$.
	Let $M$ be the middle term of a nonzero extension $0\to L\to M\to N\to 0$.
	It is easy to check that $L$ is the unique subrepresentation of $M$ of dimension $(1,1)$, which is not in ${\V}(M)=\{(0,0),(0,2),(2,2)\}$.
\end{remark}

\begin{proposition} \label{P:vhoms} $(1).$ If $L$ is a vertex subrepresentation of $M$, then 
	$\Hom(L,M/L)=0$. \\
$(2).$ If $M$ is $\Ec$-rigid, then any $L\subset M$ satisfying $\Hom(L,M/L)=0$ must be $\Ext$-rigid.
\end{proposition}
\begin{proof} (1). The fact that $\Hom(L,M/L)=0$ follows from Theorem \ref{T:torsion} and Lemma \ref{L:vertex}.
	
(2). By (1) and (2) of Lemma \ref{L:E} we have that $\check{\E}(L,M)=0$ and thus $\Ext^1(L,M)=0$.
	From the long exact sequence 
	$$0=\Hom(L,M/L) \to \Ext^1(L,L) \to \Ext^1(L,M)=0,$$
	we have that $\Ext^1(L,L)=0$.
\end{proof}

\begin{remark} (1). The dual argument shows that if $M$ is $\E$-rigid, then $M/L$ is $\Ext$-rigid.\\	
(2). We cannot replace the $\Ext$-rigidity by the $\E$ or $\check{\E}$-rigidity in the conclusion of the above lemma.
Consider $M$ as in Example \ref{ex:QPfaces}, which is $\E$-rigid since it is obtained from a sequence of mutations.
	Now $(2,1,2,1)$ is a vertex of $\N(M)$ but it is not hard to check that the corresponding vertex subrepresentation is $\Ext$-rigid but not $\E$-rigid.
	Even worse, by an algorithm in \cite{Ft} there is no general presentation whose cokernel has dimension vector $(2,1,2,1)$. 
\end{remark}

\begin{conjecture} \label{conj:vertex} Suppose that $M$ is $\E$-birigid and $L\hookrightarrow M$. Then $L$ is a vertex subrepresentation if and only if $\Hom(L,M/L)=0$.
\end{conjecture}


Now we focus on the setting of acyclic quivers. We remind readers that in this setting the notions of $\Ext$-rigidity, $\E$-rigidity and $\Ec$-rigidity all coincide.
Unless otherwise stated, we always set $\beta = \alpha-\gamma$ below.
We recall the notation $\hom(\gamma,\beta)$ and $\ext(\gamma,\beta)$ in Remark \ref{r:HE0}. 
We will write $\gamma\perp \beta$ if $\hom(\gamma,\beta)=\ext(\gamma,\beta)=0$.

\begin{lemma}[{\cite[Lemma 3.2, Theorem 3.3]{S2}}] \label{L:subrep} {\ }\begin{enumerate}
		\item A	general representation of dimension $\beta+\gamma$ contains a subrepresentation of dimension $\gamma$ if and only if $\ext(\gamma,\beta)=0$.
		\item The tangent space of $\Gr_\gamma(M)$ (as a scheme) at $L$ is isomorphic to $\Hom(L,M/L)$.
	\end{enumerate}
\end{lemma}
\noindent By the construction in \cite[Section 3]{S2} and Bertini's theorem,  $\Gr_\gamma(M)$ is generically reduced. 
For the rest part of this section, we will only deal with general representations, so part (2) will always be applicable.

\begin{lemma} \label{L:chipos} Let $M$ be a rigid representation of an acyclic quiver. Then $\chi(\Gr_{\gamma}(M))>0$ if $\Gr_{\gamma}(M)$ is non-empty.
\end{lemma}
\begin{proof} The quantum positivity (\cite[Theorem 3.2.6]{Q}, \cite[Corollary 5.2]{Fc1}) implies that the Poincar\'{e} polynomial $p$ of $\Gr_\gamma(M)$ is a polynomial with nonnegative coefficients. The evaluation of $p$ at $1$ gives the Euler characteristic of $\Gr_\gamma(M)$.
So $\chi(\Gr_{\gamma}(M))>0$ if $\Gr_{\gamma}(M)$ is non-empty.
\end{proof}

\begin{theorem} \label{T:acyclicv}  Let $M$ be an $\alpha$-dimensional rigid representation of an acyclic quiver. Then $\gamma$ is a vertex of ${\N}(M)$ if and only if $\gamma \perp \beta$.
\end{theorem}

\begin{proof} Due to Proposition \ref{P:vhoms} and Lemma \ref{L:subrep}.(1) we remain to show $\Leftarrow$. 
Suppose that $\gamma \perp \beta$ (so $\gamma\in\N(M)$ by Lemma \ref{L:subrep}.(1)) but $\gamma$ is not a vertex.
Let $\V(M) = \{\gamma_i\}_i$. Since $\gamma\in \N(M)$, there is some $n\in \mb{N}$ such that $n\gamma$ is a positive integral combination of vertices, say $n\gamma= \sum_i n_i \gamma_i$ ($n=\sum_i n_i$). 
By Lemma \ref{L:directsum} $F_{nM} = F_M^n$. So by Lemma \ref{L:chipos} $\chi(\Gr_{n\gamma} (nM))$ must be greater or equal to a multinomial coefficient, which is greater than $1$.	
Note that $nM$ is rigid since $M$ is rigid.
		
By Proposition \ref{P:vhoms} there is a rigid representation of dimension $\gamma$ or $\beta$.
Then a general representation of dimension $n\gamma$ or $n\beta$ is rigid.
It is well-known (e.g. \cite[Lemma 4.2]{DW2}) that in this case $n\gamma \circ n\beta = 1$ (see \cite[Definition 2.5]{DW2} for the meaning of $\gamma \circ \beta$).
By \cite[Theorem 2.10]{DW2} this is equivalent to that $\Gr_{n\gamma}(nM)$ is a point. A contradiction.
\end{proof}

Below we provide another proof based on the functor $t_\dtb$.
\begin{proof} Due to Proposition \ref{P:vhoms} and Lemma \ref{L:subrep}.(1) we remain to show $\Leftarrow$. 
	Since $\ext(\gamma,\beta)=0$, a general representation $M\in \rep_\alpha(Q)$ has a subrepresentation $L\in \rep_\gamma(Q)$.
	Moreover we can assume that $L$ is general in $\rep_\gamma(Q)$ and $N=M/L$ is general in $\rep_\beta(Q)$ (\cite[Lemma 7.13]{Ft}).
	From the exact sequence	$0\to \Hom(L,L) \to \Hom(L,M) \to \Hom(L,N)=0$
	we get $\hom(L,L) = \hom(L,M)$. Moreover, we have that $\Ext(L,M)=0$.
	
	We choose the weight $\delta = \innerprod{\gamma,-}$. 
	Then $\hom(\delta,L) = \hom(L,L)$.	
	We will show $\dv t_\delta(M) = \gamma$ so that $\gamma$ is a vertex by Lemma \ref{L:vertex}.	
	Suppose that there is another exact sequence $0\to L' \to M \to N'\to 0$. Let $\gamma'=\dv L'$.	
	According to the definition of the functor $t_\delta$, it suffices to show that either $\innerprod{\gamma,\gamma'}=\delta(\gamma')<\hom(\delta,L)$ or $L' \supsetneq L$.
	From the long exact sequence
	$$0\to \Hom(L,L') \to \Hom(L,M) \to \Hom(L,N') \to \Ext(L,L') \to \Ext(L,M)=0,$$
	we get $\innerprod{\gamma,\gamma'}=\hom(L,M) - \hom(L,N') = \hom(L,L) - \hom(L,N')$.
	If $\hom(L,N')>0$, then $\innerprod{\gamma,\gamma'} < \hom(L,L)$.
	If $\hom(L,N')=0$, then $L'$ must contain $L$ because otherwise the composition $L\hookrightarrow M \twoheadrightarrow N'$ is nonzero.
\end{proof}

\begin{remark} \label{R:acyclicv} The theorem is wrong without the rigidity assumption.
	Consider the $3$-arrow Kronecker quiver $\Kronthree{1}{2}$. Let $\alpha=(2,3)$ and $\gamma=(1,2)$.
	Let $M$ be a general representation of dimension $\alpha$.
	Then $\N(M)$ is the convex hull of $(0,0),(0,3),(2,3)$ so $\gamma$ is not a vertex.
	But it is easy to check that $(1,2) \perp (1,1)$.
\end{remark}

\begin{corollary} \label{C:Fvertex} Let $M$ be a rigid representation of an acyclic quiver. 
	Then $\gamma$ is a vertex of $\N(M)$ if and only if $\chi_\gamma(\Gr_\gamma(M))=1$.
\end{corollary}
\begin{proof} By Theorem \ref{T:onept} it suffices to show $\Leftarrow$. It is well-known (e.g. \cite{S2}) that $\Gr_\gamma(M)$ is smooth for a rigid representation $M$.
If $\chi(\Gr_\gamma(M))=1$, then the only possible nonzero homology group of $\Gr_\gamma(M)$ is the zeroth homology, and hence by Poincar\'{e} duality, $\Gr_\gamma(M)$ is $0$-dimensional. 
By Lemma \ref{L:subrep}, we have that $\gamma \perp \beta$. The claim follows from Theorem \ref{T:acyclicv}.
\end{proof}

\noindent Note that Corollary \ref{C:Fvertex} solves Conjecture \ref{conj:onept} for rigid representations of an acyclic quiver. Moreover, it suggestions the following conjecture.

\begin{conjecture} \label{conj:Fvertex} Let $F=\sum_{\gamma}c_\gamma \b{y}^\gamma$ be the $F$-polynomial of a cluster monomial (of any cluster algebra).
	Then $\gamma$ is a vertex of the Newton polytope of $F$ if and only if $c_\gamma=1$.
\end{conjecture}
\noindent Conjecture \ref{conj:Fvertex} generalizes the constant term conjecture (\cite[Conjecture 5.4]{FZ4}). Corollary \ref{C:Fvertex} settled this conjecture in the acyclic quiver case. 
The general skew-symmetric case would follow from Conjecture \ref{conj:vertex}.

\section{Saturation} \label{S:saturation}
\begin{definition}[\cite{MTY}] We say the support of a polynomial $F=c_\gamma \b{y}^\gamma$ {\em saturated} if $c_\gamma\neq 0$ for any lattice points in the Newton polytope of $F$.
	Similarly we say the sub-lattice of a representation $M$ {\em saturated} if for any lattice point $\gamma\in \N(M)$ there is some $\gamma$-dimensional subrepresentation $L\hookrightarrow M$.
\end{definition}
\noindent It is easy to see that the saturation of the support of $F_M$ implies the saturation of the sub-lattice of $M$.
Note that in general the saturation of the support of $F_M$ does not imply the saturation of the sub-lattice of $M$.

\begin{theorem}\label{T:saturation} Let $M$ be a rigid representation of an acyclic quiver. Then both the sub-lattice of $M$ and the support of $F_M$ are saturated.
\end{theorem}
\begin{proof} Let $\gamma_0$ be a vertex of $\N(M)$ and $\beta_0$ be any vertex of $\check{\N}(M)$.
	Since $\ext(M,M)=0$, we have that $\ext(\gamma_0,\beta_0)=0$.
	Corollary \ref{C:HE0} implies that all dimension vectors $\beta$ such that $\ext(\gamma_0,\beta)=0$ are given by convex polyhedral conditions. 
	So $\ext(\gamma_0,\beta)=0$ for any lattice point $\beta$ in $\check{\N}(M)$.
	For the same reason $\ext(\gamma,\beta)=0$ for any lattice point $\gamma$ in $\N(M)$.
	By Lemma \ref{L:subrep}.(1) the sub-lattice of $M$ is saturated.
	
	By the quantum positivity of $\Gr_\gamma(M)$ (Lemma \ref{L:chipos}), each $\Gr_\gamma(M)$ has positive Euler characteristic. Hence the support of $F_M$ is saturated as well.
\end{proof}

This is equivalent to say that for a cluster algebra of some acyclic quiver, the support of $F$-polynomial of any cluster monomial is saturated. 
In general we conjecture that ``acyclic quiver" can be dropped.

\begin{conjecture} \label{conj:saturation} The support of the $F$-polynomial of a cluster monomial (of any cluster algebra) is saturated.
\end{conjecture}
\noindent When the cluster algebra is skew-symmetric (but not necessarily acyclic), by the categorification \cite{DWZ2} we can associate a Jacobian algebra to model the cluster algebra. 
Then the conjecture is equivalent to that the support of $F_M$ is saturated for any mutation-reachable representation $M$.

\begin{remark} \label{R:saturation} The rigidity assumption is necessary. Consider the $3$-arrow Kronecker quiver and $M$ a general representation of dimension $\alpha=(3,4)$.
Then $\gamma=(2,3)$ is a lattice point in $\N(M)=\op{conv}\{(0,0),(0,4),(3,4)\}$ but $\ext(\gamma,\alpha-\gamma)=1$.	
So neither the support of $F_M$ nor the sub-lattice of $M$ is saturated.
\end{remark}

\section{Restriction to Facets} \label{S:faces}
\subsection{The $\delta$-stable Reduction of a Representation} \label{ss:stablered}
Let $\mc{V}=\{V_1,V_2,\dots,V_k\}$ be a set of pairwise non-isomorphic $\delta$-stable representations.
Let $\mc{F}(\mc{V})$ be the full subcategory of $\rep(A)$ whose objects are those representations admitting a finite filtration with factors in $\mc{V}$.
Recall from \cite[Theorem 1(d)]{R} (or prove directly from Definition \ref{D:King}) that $\Hom(V_i,V_j)=\mb{C}$ if $i=j$, otherwise $\Hom(V_i,V_j)=0$.
In particular,  $\mc{F}(\mc{V})$ is an abelian subcategory.

Let $\op{tw}(\mc{V})$ be the category of {\em twisted stalks} over the $A_\infty$-category $\mc{V}$ \cite{Ke2} (here we also use $\mc{V}$ to denote the $A_\infty$-category with objects $\{V_1,V_2,\dots,V_k\}$ and morphisms $\Ext_A^*(V_i, V_j)$).
Roughly speaking, a twisted stalk can be described as a pair $(\mb{V},\epsilon)$ where $\mb{V}$ is a sequence of factors from $\mc{V}$ and 
$\ep$ is an upper triangular matrix with entries from $\Ext^1(V_i,V_j)$ satisfying the Maurer-Cartan equation.

\begin{theorem}[{\cite[Proposition 2.3]{Ke2},\cite{KKO}}] The category $\mc{F}(\mc{V})$ is equivalent to the category of twisted stalks $\op{tw}(\mc{V})$.
\end{theorem}

We introduce the following quiver with relations $(Q_{\mc{V}}, r_{\mc{V}})$:
\begin{enumerate}
	\item The vertices of $Q_{\mc{V}}$ are in bijection with elements in $\mc{V}$;
	\item The arrows from $V_i$ to $V_j$ are in bijection with a basis in $\Ext_A^1(V_i,V_j)^*$;
	\item The relations $r_{\mc{V}}$ are given by the image of the dual of the $A_\infty$-structure maps
	$$m_n: \Ext_A^1(V,V)^{\otimes n} \to \Ext_A^2(V,V).$$
\end{enumerate}
Let $A_{\mc{V}}$ be the completed path algebra $\widehat{\mb{C}Q_{\mc{V}}}$ modulo the closed ideal generated by $r_{\mc{V}}$.

\begin{lemma}[{\cite[Proposition 6.3]{KKO}}] \label{L:rho} There is a functor $ \op{tw}(\mc{V}) \to \rep(A_{\mc{V}})$,
	and thus a functor $\rho_{\mc{V}}: \mc{F}(\mc{V}) \to \rep(A_{\mc{V}})$, which is bijective on objects.
	Moreover, $U$ is a subrepresentation of $W$ in $\mc{F}(\mc{V})$ if and only if $\rho_{\mc{V}}(U)$ is a subrepresentation of $\rho_{\mc{V}}(W)$.
\end{lemma}

\begin{remark} We expect that the functor $\rho_{\mc{V}}$ is an equivalence.
In general, according to \cite{Ke2,KKO} the category $\mc{F}(\mc{V})$ is equivalent to a category which has the same objects as $\rep(A_{\mc{V}})$ but may contain extra morphisms.
In our special situation where $\mc{V}$ consists of $\delta$-stable representations, we believe there are no extra morphisms.

Moreover, since $\hom(V_i, V_j) = \delta_{ij}$, the construction of this functor in \cite{KKO} implies that it sends $\mc{V}$ bijectively to the set of $1$-dimensional simple representations in $\rep(A_{\mc{V}})$.
\end{remark}


Recall that each $\delta$-semistable representation $W$ has a filtration:
$$0=W_0 \subset W_1 \subset W_2 \subset \cdots \subset W_l=W$$
such that each factor $W_{i}/W_{i-1}$ is $\delta$-stable.
This is the Jordan-Holder filtration of $W$ in the category of $\delta$-semistable representations.
It is well-known that the Jordan-Holder filtration is not unique but each factor $V_i$ with multiplicities is uniquely determined by $W$.
We denote by $\mc{V}_\delta(W)$ the set of all $\delta$-stable factors in such filtration of $W$.

\begin{definition}
We call $\rho_{\mc{V}_\delta(W)}(W)$ {\em the $\delta$-stable reduction of $W$}, 
and abbreviate it as $\rho_{\delta}(W)$.
\end{definition}
\noindent We remark that a certain geometric construction related to $\rho_\delta$ for quivers without relations was considered in \cite{AB}.
Also note that $\rho_{\delta}$ is not a functor because its target will vary according to $W$.
We also abbreviate $A_{\mc{V}_\delta(W)}$ as $A_{\delta,W}$.

Let $m \in \mb{Z}^{(Q_{\mc{V}})_0}$ be the multiplicity vector of $U\in\mc{F}(\mc{V})$, 
i.e., each $V_i$ appears $m(i)$ times as a factor in the $\delta$-stable filtration of $U$. 
Let $\Gr_{m}^{\delta}(W)$ be the variety parametrizing $\delta$-semistable subrepresentations of $W$ with multiplicity vector $m$ (in $\mc{V}_\delta(W)$).
As a simple corollary of Lemma \ref{L:rho}, we have that
\begin{corollary} \label{C:red} The variety $\Gr_m^\delta(W)$ is isomorphic to $\Gr_{m}( \rho_\delta (W) )$.
\end{corollary}

Let us make a connection with the constructions in \cite{GL,S1}.
We now suppose that $A=\mb{C}Q$ has no relations.  
We assume that $\ep$ is a weight such that either $\ep=-e_i$ or $\Coker(\ep)$ is an exceptional representation $E$.
Recall that a representation $E$ is called {\em exceptional} if $\Hom(E,E)=\mb{C}$ and $\Ext(E,E)=0$.
This implies that $\dv E$ is a {\em real Schur root}.
By Remark \ref{r:n=1}, the category of $\ep$-semistable representations is nothing but the right orthogonal subcategory 
$$E^\perp:=\{M\in\rep(Q) \mid \Hom(E,M)=\Ext(E,M)=0 \}.$$

\begin{lemma}[{\cite{GL,S1}}] \label{L:ortho} The embedding $\ss{\ep}{Q} \hookrightarrow \rep(Q)$ has a left adjoint $p_{\ep}: \rep(Q) \to \ss{\ep}{Q}$.
	The category $\ss{\ep}{Q}$ is equivalent to $\rep(Q_\ep)$ where $Q_\ep$ is a quiver with one vertex less than $Q$.
\end{lemma}
\begin{remark} In Schofield's original paper, he only deals with the case when $\ep\neq -e_i$ but it is rather trivial to incorporate this case.
Moreover, there is also a right adjoint $\check{p}_{\ep}: \rep(Q) \to \ss{\ep}{Q}$ which can be constructed using the Auslander-Reiten duality. 	
\end{remark}

\noindent Let $\mc{V}$ be the set of simple objects in $\ss{\ep}{Q}$. Then the algebra $\mb{C}Q_{\mc{V}}$ is just the path algebra $\mb{C}Q_\ep$.
The functor $\rho_{\mc{V}}$ is nothing but the equivalence $\ss{\ep}{Q} \to \rep(Q_\ep)$. 
So we denote the composition of $p_\ep$ (resp. $\check{p}_\ep$) with the equivalence by $\pi_{Q_\ep}$ (resp. $\check{\pi}_{Q_\ep}$): $\rep(Q)\to \rep(Q_\ep)$.
The above comments and notions will be used in Corollary \ref{C:acyclicf}.

\subsection{General cases}
From now on, we assume WLOG that the representation $M$ is supported on each vertex of the quiver so that its Newton polytope is full-dimensional (because $M$ is filtered by $1$-dimensional simples).
Otherwise, we can always work with some subalgebra $eAe$ where $e$ is an idempotent of $A$.
Suppose that $F_M = \sum_{\gamma} c_\gamma \b{y}^\gamma$.
In this subsection, we study the restriction of $F_M$ to some face ${\sf \Lambda}$ of $\N(M)$.
By definition, it is 
$$\sum_{\gamma\mid \gamma\in \sf\Lambda} c_\gamma \b{y}^\gamma.$$

We denote by $\pi_{\delta}$ the composition of the stabilization functor $\dtb^\perp$ with the $\delta$-stable reduction $\rho_\delta$.
For the same reason as $\rho_\delta$, $\pi_{\delta}$ is not a functor.
Let $\iota_{\mc{V}}$ be the linear map $K_0(\rep(A_{\mc{V}})) \to K_0(\rep(A))$ induced by ${e}_i \mapsto \dv V_i$.
So for each $W$ we have a linear map $\iota_{\mc{V}_\delta(W)}$. We will abbreviate it as $\iota_{\delta}$.
We keep in mind that the definition of $\iota_\delta$ will depend on $W$.
Each $\mb{Z}$-linear map $\iota: \mb{Z}^m\to \mb{Z}^n$ which restricts to $\mb{N}^m\to \mb{N}^n$
induces a {\em monomial change of variables} $\mb{C}[x_1,x_2,\dots,x_m] \to \mb{C}[y_1,y_2,\dots,y_n]$ given by the linear extension of
$$\b{x}^\alpha \mapsto \b{y}^{\iota(\alpha)}.$$
Recall the functors $t_\dtb$ and $\tc_\dtb$ in Section \ref{ss:torsion}.
\begin{theorem} \label{T:faces} Let $M\in\rep(A)$. Let $\delta$ be the outer normal vector of some facet of the Newton polytope $\N(M)$.
Then the restriction of $F_M$ to this facet is given by 
$$ \b{y}^{\dv t_{\dtb}(M)} \iota_{\delta} \left( F_{\pi_{\delta} (M) } \right).$$
\end{theorem}

\begin{proof} The facet $\sf \Lambda$ is given by the convex hull of the dimension vectors of all elements in $\mc{L}(\delta,M)$.
	Recall the torsion-pair sequence 
	$$0\to t_\dtb(M) \to M \to f_\dtb(M) \to 0.$$
	Every subrepresentation $L$ in $\mc{L}(\delta,M)$ lie between $L_{\min}=t_\dtb(M)$ and $L_{\max}=\check{t}_\dtb(M)$. 
	So the natural map $L\mapsto L/L_{\min}$ give an isomorphism $\Gr_{\gamma} (M) \cong \Gr_{\gamma-\gamma_{\min}}(L_{\max}/L_{\min})$ where $\gamma_{\min}=\dim L_{\min}$.	
	Moreover, $L/L_{\min}$ is a $\delta$-semistable subrepresentation of $\dtb^\perp(M) = L_{\max}/L_{\min}$.
	So $\Gr_{\gamma'}(L_{\max}/L_{\min})$ is stratified according to the multiplicity $m$:
	\begin{equation} \label{eq:mstr} \Gr_{\gamma'}(L_{\max}/L_{\min}) = \bigsqcup_{m\mid \iota_{\delta}(m)=\gamma'} \Gr_{m}^\delta(L_{\max}/L_{\min}). \end{equation}
Finally we recall from Corollary \ref{C:red} that $\Gr_m^\delta( \dtc^\perp(M) ) \cong \Gr_{m}( \pi_\delta (M) )$.
Now we have that 
\begin{align*} \b{y}^{\gamma_{\min}} \iota_{\delta} (F_{\pi_\delta(M)}) &= \b{y}^{\gamma_{\min}} \iota_\delta\Big(\sum_{m} \chi\big(\Gr_m^\delta(L_{\max}/L_{\min})\big) \b{y}^m \Big)  && \text{by Corollary \ref{C:red}} \\
 &=  \b{y}^{\gamma_{\min}} \sum_{\gamma\in \sf\Lambda} \sum_{m \mid \iota_\delta(m)=\gamma-\gamma_{\min}} \chi\big(\Gr_m^\delta(L_{\max}/L_{\min})\big) \b{y}^{\gamma - \gamma_{\min}}  \\
 &=  \sum_{\gamma\in\sf\Lambda} \chi\Bigg( \bigsqcup_{m \mid \iota_\delta(m)=\gamma-\gamma_{\min}} \Gr_m^\delta(L_{\max}/L_{\min})\Bigg) \b{y}^{\gamma}  \\
 &=  \sum_{\gamma \in \sf\Lambda } \chi\big(\Gr_{\gamma-\gamma_{\min}}(L_{\max}/L_{\min})\big) \b{y}^{\gamma} && \text{by \eqref{eq:mstr}} \\
 &=  \sum_{\gamma \in \sf\Lambda } \chi\left(\Gr_\gamma(M)\right) \b{y}^{\gamma}.
\end{align*}
\end{proof}

\begin{question} \label{q:v-1} Suppose that a representation $M$ is supported on each vertex of $Q$.
	Let $\delta$ be an outer normal vector of $\N(M)$.
	Suppose that $\pi_{\delta}(M)$ is supported on the quiver $Q'$, which is also the quiver of $A_{\delta,\dtb^\perp(M)}$.
	It is rather easy to show that $Q'$ has at least $|Q_0|-1$ vertices.	
    We have an example where $Q'$ has more than $|Q_0|-1$ vertices.	
    Is it true that if $M$ is the cokernel of a general presentation, then the quiver $Q'$ has exactly $|Q_0|-1$ vertices?	
\end{question}

A main source of our examples come from the theory of quivers with potentials \cite{DWZ1,DWZ2}.
One advantage is that we have an algorithm to compute the $F$-polynomial for mutation-reachable representations.
The algorithm was implemented in \cite{Ke}.

\begin{example} \label{ex:QPfaces} Consider the quiver $\cyclicfourone{1}{4}{2}{3}{a}{b}{c}$ with potential $abc$.
	Let $M$ be the representation obtained from the sequence of mutations $(3,4,1,2)$. 
	Its $\delta$-vector is $(-1,1,1,0)$, and its dimension vector is $(2,1,3,1)$.
	We find its $F$-polynomial
\begin{align*}
	F_M = 1 &+ y_3  + y_3y_4 +2y_1  + 4y_1y_3 + 2y_1y_3y_4 + 2y_1y_3^2 + 2y_1y_3^2y_4 + y_1^2   + 3y_1^2y_3\\
	& + y_1^2y_3y_4 + 3y_1^2y_3^2 + 2y_1^2y_3^2y_4 + y_1^2y_3^3 + y_1^2y_3^3y_4 + y_1^2y_2y_3^2y_4 + y_1^2y_2y_3^3y_4. 
\end{align*}

In the following table, we list all $7$ facets with their normal vectors $\delta$.
Other data includes $f_M(\delta),\ \fc_M(-\delta)$, and the dimension vectors of $t_\dtb(M)$ and $\tc_\dtb(M)$. 
The default for $\dv t_\dtb(M)$ is $0$ and the default for $\dv \tc_\dtb(M)$ is $\dv M$.
\begin{center}
\begin{tabular}{l*{5}{c}r}
	No.              & $\delta$ & $f_M(\delta)$ & $\fc_M(-\delta)$ & $\dv t(M)$ & $\dv \tc(M)$  \\
	\hline
	1 & $(-1,2,0,0)$ & 0 & 0 & - &  -  & \\
	2 & $(0,1,0,-1)$ & 0 & 0 & - &  -  & \\
	3 & $(0,-1,0,0)$ & 0 & 1 & - & $(2,0,3,1)$ & \\ 
	4 & $(0,1,-1,1)$ & 0 & 1 & - &  $(2,1,2,1)$ &   \\ %
	5 & $(1,0,0,0)$ & 2 & 0 & $(2,0,0,0)$ &  -  &   \\ %
	6 & $(0,0,0,1)$ & 1 & 0 & $(0,0,1,1)$ &  - &  \\
	7 & $(-1,0,1,0)$ & 1 & 0 & $(0,0,1,0)$ &  -  & 
\end{tabular}
\end{center}


\noindent Here are the restrictions of $F_M$ to these facets:
\begin{align*} 
&  1+y_3 + y_3y_4 + y_1^2y_2y_3^2y_4 + y_1^2y_2y_3^3y_4 \\
& (1+y_3)(1 + 2y_1 + y_1^2  + 2y_1y_3 + 2y_1^2y_3 + y_1^2y_3^2 + y_1^2y_2y_3^2y_4)\\
& (1+y_1 + y_1y_3)^2(1+y_3 + y_3y_4) \\
& 1 + 2y_1 + y_1^2 +  y_3y_4  + 2y_1y_3y_4 + y_1^2y_3y_4 + y_1^2y_2y_3^2y_4 \\
& y_1^2(1+y_3)(1+2y_3+ y_3^2 + y_3y_4 + y_3^2y_4  + y_2y_3^2y_4) \\
&  y_3y_4(1 + 2y_1 + 2y_1y_3  + y_1^2 + 2y_1^2y_3 + y_1^2y_3^2 + y_1^2y_2y_3 + y_1^2y_2y_3^2)\\
& y_3(1+y_4 + 2y_1y_3  + 2y_1y_3y_4  + y_1^2y_3^2 + y_1^2y_3^2y_4 + y_1^2y_2y_3^2y_4)
\end{align*}

Let $T$ be the unique indecomposable representation of dimension $(2,1,1,0)$, and $T_{i,j}$ be the unique indecomposable representation of dimension $e_i+e_j$.
We list below the quiver of $A_{\delta, \dtb^\perp(M)}$ and the dimension vector of $\pi_{\delta} (M)$, and the polynomials $F_{\pi_{\delta} (M)}$.
The variable $y_i$ corresponds to the vertex $S_i$. The variable $y_{ij}$ corresponds to the vertex $T_{ij}$, and $y_{t}$ corresponds to the vertex $T$.  
The relation for the first quiver is also given by (a generic) potential. There are no relations for the rest of the quivers. 
The facet no. 3,\ 5, and 6 are rather trivial so we left them to readers.
\begin{align*}
&1 && \oneonetwo{T}{S_3}{S_4}  &&  \oneonetwo{1}{2}{1} && 1+y_3 + y_3y_4 + y_ty_3y_4 + y_ty_3^2y_4 \\
&2 && \twoone{T_{24}}{S_3}{S_1}  &&  \twoone{1}{3}{2} &&  (1+y_3)(1+2y_1+ y_1^2  + 2y_1y_3 + 2y_1^2y_3 + y_1^2y_3^2 + y_1^2y_{24}y_3^2) \\  
&4 && \twoonea{S_1}{T_{23}}{T_{34}}  &&  \twoonea{2}{1}{1} && 1+2y_1 + y_1^2 + y_{34} + 2y_1y_{34} + y_1^2y_{34} + y_1^2y_{23}y_{34}\\
&7 && \twoonea{T_{13}}{S_2}{S_4}  &&  \twoonea{2}{1}{1} && 1+2y_{13} + y_{13}^2 + y_{4} + 2y_{13}y_{4} + y_{13}^2y_{4}  + y_{13}^2y_{2}y_{4}
\end{align*}

\end{example}

\begin{example}\label{ex:heneq0}  It is possible that both $f_M(\delta)$ and $\fc_M(-\delta)$ do not vanish.
	In such situation, both	$t_\dtb(M)$ and $\tc_\dtb(M)$ must be nontrivial.
	We shall see in Theorem \ref{T:Newton} that this is impossible for a general representation of a quiver.
	For the same quiver with potential, let $M$ be the representation obtained from the sequence of mutations $(2,3,4,1,2,3)$.
	Its $\check{\delta}$-vector is $(1,-1,1,1)$, and its dimension vector is $(2,5,2,2)$.
	One can easily check that $\N(M)$ has a facet with normal vector $\delta=(-1,0,1,0)$, and $f_M(\delta)=\fc_M(-\delta)=1$.
	We find that $\dv t_\dtb(M)= (0,0,1,0)$ and $\dv \tc_\dtb(M)=(1,4,2,2)$.
	The restriction of $F_M$ to this facet is  
	$$y_3(1+y_4 + y_2y_4)^2(1 + y_1y_3 + 2y_1y_2y_3 + y_1y_2^2y_3).$$
	\begin{align*}
	& \twoonea{T_{13}}{S_2}{S_4}  &&  \twoonea{1}{4}{2} &&  (1+y_4 + y_2y_4)^2(1 + y_{13} + 2y_{13}y_2 + y_{13}y_2^2)
	\end{align*}
	
\end{example}

\begin{remark}[on arbitrary faces] \label{R:anyfaces} Theorem \ref{T:faces} allows us to interpret the restriction of $F_M$ to any (not necessarily codimension-1) face ${\sf \Lambda}$ by induction.
	Alternatively one can choose some $\delta$ such that ${\sf \Lambda}$ is supported on $H=\{\gamma \mid \delta(\gamma) =h\}$ for some $h$ and $H \cap \N(M) = {\sf \Lambda}$. 
	Then the same proof gives the same formula.
\end{remark}

\subsection{Acyclic Case}
We come back to the setting of quivers without oriented cycles. 
In this subsection we mostly deal with a general representation $M$ of an acyclic quiver $Q$.
It is good to keep in mind that $f_M(\delta)=\hom(\delta,M)$ (see Remark \ref{r:n=1}).

Recall the sets ${\sf \Delta}_0(M)$ and ${\sf \Delta}_1(M)$ (see Section \ref{ss:Ftrop}). 
By Remark \ref{r:n=1}.(2) we have that \begin{align} \label{eq:Delta01}
{\sf \Delta}_0(M)&=\{\delta\in \mb{Z}^{Q_0} \mid {\hom}(\delta,M)=0 \};\\
{\sf \Delta}_1(M)&=\{\delta\in \mb{Z}^{Q_0} \mid {\ext}(\delta,M)=0 \}.
\end{align}
Let ${\sf R}_i(M)$ be the set of extremal rays in the cone $\mb{R}_{\geq 0}{\sf \Delta}_i(M)$ ($i=0,1$).
Throughout we will represent an element in ${\sf R}_i(M)$ by an indivisible integral vector.
Such a vector must be an indecomposable $\delta$-vector.

\begin{theorem} \label{T:Newton} Let $M$ be a general representation in $\rep_\alpha(Q)$. Then the outer normal vectors of $\N(M)$ are precisely ${\sf R}_0(M) \cup {\sf R}_1(M)$, 
	and $\N(M)$ can be presented as
$$\{ \gamma\in\mb{R}^{Q_0} \mid \delta_0(\gamma) \leq 0 \text{ for $\delta_0 \in {\sf R}_0(M)$ and } {\delta}_1(\alpha-\gamma) \geq 0 \text{ for ${\delta}_1 \in {\sf R}_1(M)$} \}.$$
\end{theorem}

\begin{proof} Let us denote the above given set by $\N'(M)$. By Corollary \ref{C:allfaces} we knew that $\N(M) \subseteq \N'(M)$.
Since ${\sf R}_i(M)$ is the set of extremal rays in the cone $\mb{R}_{\geq 0}{\sf \Delta}_i(M)$ ($i=0,1$), the set $\N'(M)$ is equal to
$$\{ \gamma \in\mb{R}^{Q_0} \mid \delta_0(\gamma) \leq 0 \text{ for $\delta_0 \in {\sf \Delta}_0(M)$ and } {\delta}_1(\alpha-\gamma) \geq 0 \text{ for ${\delta}_1 \in {\sf \Delta}_1(M)$} \}.$$	
Note that for ${\delta}_1 \in {\sf \Delta}_1(M)$, ${\delta}_1(\alpha-\gamma) \geq 0$ is equivalent to that ${\delta}_1(\gamma) \leq \hom(\delta_1,M)$.
Next we will show that $\N(M) \supseteq \N'(M)$.

Let $\delta$ be any outer normal vector of $\N(M)$. By Corollary \ref{C:indnormal} $\delta$ is generically indecomposable.
Then either $\delta=-{e}_i$ or $\PHom(\delta)$ is generically injective by \cite[Theorem 6.3.1]{IOTW}.
Clearly $-e_i \in {\sf \Delta}_0(M)$ (but $\notin {\sf \Delta}_1(M)$).
For the latter case, let $C$ be the general cokernel of $\PHom(\delta)$. 
Let $\delta_\alpha=\innerprod{\alpha,-}$ be the $\delta$-vector corresponding to $\alpha$.
Consider the torsion-pair sequence of $C$ given by the functor $(t,f)$ for the weight $\delta_\alpha$ 
$$0\to t(C) \to C \to f(C)\to 0.$$
	Then by the dual of Lemma \ref{L:sub=h} (see Remark \ref{R:dual}) we have that
	 $$\Hom(t(C),\delta_\alpha)=0 \text{ and } \Ec(f(C),\delta_\alpha)=0.$$
So 
\begin{equation} \label{eq:van_tf} \Hom(t(C),M)=\Ext(f(C),M)=0.
\end{equation}
Thus we get a long exact sequence
	$$0\to \Hom(f(C),M) \to \Hom(C,M) \to \Hom(t(C),M) \to \Ext(f(C),M) = 0.$$
Let $\delta_t$ and $\delta_f$ be the $\delta$-vector of the minimal presentation of $t(C)$ and $f(C)$.
Since $\PHom(\delta)$ is generically injective, we have that $\delta = \delta_t + \delta_f$.
By \eqref{eq:van_tf} $\delta_t \in {\sf \Delta}_0(M)$ and $\delta_f \in {\sf \Delta}_1(M)$.

Let $\gamma\in \N'(M)$, so it satisfies the inequalities
\begin{equation} \label{eq:N'} \delta_{t}(\gamma) \leq \hom(\delta_{t},M) \text{ and } \delta_{f}(\gamma) \leq \hom(\delta_{f},M). \end{equation}
We claim that the inequality $\delta(\gamma) \leq \hom(\delta,M)$ is implied by \eqref{eq:N'} so that $\gamma \in \N(M)$.
This is because 
$$\delta(\gamma) = \delta_t(\gamma) + \delta_f(\gamma) \leq \hom(\delta_{t},M)+\hom(\delta_{f},M) \leq \hom(t(C),M)+\hom(f(C),M) = \hom(C,M).$$
We thus get the equality $\N(M)=\N'(M)$ and in particular that all outer normal vectors of $\N(M)$ are in ${\sf R}_0(M) \cup {\sf R}_1(M)$.

We remain to show that all elements in ${\sf R}_0(M) \cup {\sf R}_1(M)$ are outer normal vectors of $\N(M)$.
If $\delta\in {\sf R}_0(M)$ is not a normal vector, then ${\mc L}(\delta,M)$ has codimension $>1$ and contains $0$.
Let ${\sf \Lambda}_1$ and ${\sf \Lambda}_2$ be any two facets containing ${\mc L}(\delta,M)$.
Since ${\mc L}(\delta,M)$ contains $0$, the normal vectors of ${\sf \Lambda}_1$ and ${\sf \Lambda}_2$ are in $\Delta_0$.
Then $\delta$ is a positive combination of these two normal vectors, which contradicts $\delta$ being extremal.
A similar argument can treat the case when $\delta\in {\sf R}_1(M)$.

\end{proof}

\begin{remark} \label{R:Newton} (1). Theorem \ref{T:Newton} is false for general quivers with potentials, even for $\E$-rigid representations. Example \ref{ex:heneq0} is a counterexample.
	
(2). The intersection $\mb{R}_{\geq 0}{\sf\Delta}_0(M) \cap \mb{R}_{\geq 0}{\sf\Delta}_1(M)$ is again a polyhedral cone. Indeed, it is the section of $\mb{R}_{\geq 0}{\sf\Delta}_0(M)$ or $\mb{R}_{\geq 0}{\sf\Delta}_1(M)$ by the hyperplane $\{\delta\mid \delta(\alpha)=0\}$. When $M$ is a general representation of $Q$ of dimension $\alpha$, it is the GIT cone of $\rep_\alpha(Q)$ studied by many authors (e.g. \cite{DW2}).
	Now we see that the intersection is a common face of $\mb{R}_{\geq 0}{\sf\Delta}_0(M)$ and $\mb{R}_{\geq 0}{\sf\Delta}_1(M)$.
\end{remark}

We shall show further that if $M$ is rigid, then the rays in ${\sf R}_0(M)$ and ${\sf R}_1(M)$ are either $-{e}_i$ or correspond to real Schur roots.
We first prove two interesting lemmas.
Recall the definition of $\mc{L}(\delta,M)$ before Theorem \ref{T:maxminsub}.
In this subsection we mostly ignore all the subscripts $\dtb$ for the functors $t,f,\tc$, and $\fc$.

\begin{lemma} \label{L:perp} Let $M$ be a general representation of $Q$ of dimension $\alpha$. Then $t(M) \perp L/t(M)$ and $\tc(M)/L \perp \fc(M)$ for any $L\in \mc{L}(\delta,M)$. 
\end{lemma}
\begin{proof} We only show that $t(M) \perp L/t(M)$ because $\tc(M)/L \perp \fc(M)$ can be treated by the dual argument. 
	By Theorem \ref{T:torsion} we have that
	$$\Hom(t(M),M/t(M))=0\ \Rightarrow\ \Hom(t(M),L/t(M))=0.$$
	On the other hand, we have $\ext(\gamma,\alpha-\gamma)=0$ by Lemma \ref{L:subrep} where	$\gamma=\dv t(M)$.
	So $\innerprod{\gamma,\alpha-\gamma}=0$ and thus $\Ext(t(M),M/t(M))=0$.
	Apply the functor $\Hom(t(M),-)$ to the short exact sequence 
	$$0\to L/t(M) \to M/t(M) \to M/L\to 0$$
	we get
	$$0= \Hom(t(M),M/L)\to \Ext(t(M),L/t(M)) \to \Ext(t(M),M/t(M))=0.$$
	The vanishing of $\Hom(t(M),M/L)$ is due to Theorem \ref{T:torsion}.
	We conclude the vanishing of $\Ext(t(M),L/t(M))$ so $t(M) \perp L/t(M)$. 	
\end{proof}

\begin{remark} This is a quite surprising result. In particular, it implies that if $\delta$ is an indivisible outer normal vector of some facet of $\N(M)$ with $f_M(\delta)>0$, then the minimal element of this facet, namely $\gamma=\dv t(M)$, has weight $\innerprod{\gamma,-}$ equal to some multiple of $\delta$!
\end{remark}

\begin{lemma} \label{L:codim2} Let $M$ be a general representation in $\rep_\alpha(Q)$.
	If $\delta$ corresponds to an imaginary root and $\hom(\delta,M)>0$, then
	the convex hull of the dimension vectors in $\mc{L}(\delta,M)$ has codimension at least $2$.
	In particular, such a $\delta$-vector cannot be a normal vector of $\N(M)$.
\end{lemma}
\begin{proof} Let $L$ be any element in $\mc{L}(\delta,M)$.
	Lemma \ref{L:perp} says that $L/t(M)$ is $\delta_\gamma$-stable where $\delta_\gamma = \innerprod{\gamma,-}$ is the weight of $t(M)$.
	In the meanwhile, $L/t(M)$ is also $\delta$-semistable by Theorem \ref{T:maxminsub}.
	Since $\delta$ corresponds to an imaginary root and $\delta(\gamma) = \hom(\delta,M)>0$, $\delta$ is not a multiple of $\delta_\gamma$. 
	So $\delta$ and $\delta_\gamma$ span a subspace of dimension 2, and they are both orthogonal to $\dv L/t(M)$.
	Hence the convex hull of $\mc{L}(\delta,M)$ has dimension at most $|Q_0|-2$.
\end{proof}

\begin{lemma}\label{L:inj} Suppose that $M$ is a rigid representation of $Q$.
	Let $\ep$ be any element in ${\sf R}_1(M)$, and $E=\Coker(\ep)$.
	Then $E$ is exceptional and $hE$ is a vertex subrepresentation of $M$ where $h=\hom(E,M)$.
\end{lemma}
\begin{proof} By Theorem \ref{T:Newton}, $\ep$ is a normal vector of $\N(M)$ so $\mc{L}(\ep,M)$ has codimension $1$.
	Since $\ep$ is an indecomposable $\delta$-vector, $E$ must be a Schur representation \cite{Ka2}.
	So to show $E$ is exceptional, it suffices to show $\dv E$ is real, i.e., $\ep(\dv E)=1$.	
	If $h=0$ then $E\in {^\perp M}$. In fact, it is a simple object in $^\perp M$ so that $\dv E$ is real \cite{S1}.
	This is because otherwise there is some subrepresentation $F$ of $E$ such that both $F$ and $E/F$ are in ${^\perp M}$,
	which contradicts that $\ep\in {\sf R}_1(M)$.
	Suppose that $h>0$, then $\dv E$ is real by Lemma \ref{L:codim2}.

	We fix a basis $\{f_1,\dots, f_h\}$ of $\hom(E,M)$, and let $\phi_n$ be the map $(f_1,\dots,f_n): nE\to M$ where $n=1,2,\dots,h$.
	By Corollary \ref{C:univhomo} and Lemma \ref{L:vertex} we remain to show that the universal homomorphism $hE \to M$  is injective.
	If not injective, let $m$ be the minimal number such that $mE \xrightarrow{\phi_m} M$ is not injective.
	We claim that the kernel $K$ and image $I$ of this map must be isomorphic to $E$ and $(m-1)E$ respectively.
	Indeed, since $\Ext(E,M)=\Ext(M,M)=0$ and $K\hookrightarrow mE, I\hookrightarrow M$, we have that $\Ext(K,M)=\Ext(I,M)=0$.
	So $\delta_K$ and $\delta_I \in {\sf \Delta}_1(M)$.
	But $h\ep = \delta_K+\delta_I$ and $\ep$ is extremal, so $\delta_K$ and $\delta_I$ have to be multiples of $\ep$.
	Thus $\dv K$ and $\dv I$ are multiples of $\dv E$.
	Since $(m-1)E\xhookrightarrow{\phi_{m-1}} M$, we must have that $I \cong (m-1)E$ and $\dv K = \dv E$.
	Since $E$ is Schur, the morphism $mE \to I\cong (m-1)E$ is represented by an $(m-1)\times m$ complex matrix.
	This contradicts the fact that $\{f_1,\dots,f_m\}$ is linearly independent.
\end{proof}

\begin{corollary} \label{C:Newton} If $M$ is rigid, then the rays in ${\sf R}_0(M)$ and ${\sf R}_1(M)$ are either $-{e}_i$ or correspond to real Schur roots.
\end{corollary}
\begin{proof} We only need to prove for ${\sf R}_1(M)$ because the dual argument can deal with the rays in ${\sf R}_0(M)$.
	Note that by Remark \ref{R:dual} we have 
	$${\sf \Delta}_0(M)=\{\delta\in \mb{Z}^{Q_0} \mid {\ext}(M,-\delta)=0 \}.$$
	The statement for ${\sf R}_1(M)$ is just the content of Lemma \ref{L:inj}.
\end{proof}

The following example shows that for general quivers with potentials the outer normal vector may not correspond to an $\Ext$-rigid representation.
\begin{example} For the same quiver with potential in Example \ref{ex:QPfaces}, let $M$ be the $\E$-rigid representation with $\check{\delta}$-vector $(1,-1,2,0)$.
	This one is reachable by a sequence of mutations $(2,3,4,1,2)$.
	Then $\N(M)$ has a facet with normal vector $(0,1,-1,-1)$ whose generic cokernel is a $(1,1,0,1)$-dimensional representation,
	which is easily seen to be not $\Ext$-rigid.
\end{example}

A facet with outer normal vector $\delta$ is called a {\em 0-facet} if $\hom(\delta,M)=0$ and is called a {\em +-facet} if $\hom(\delta,M)>0$.
Let $\ep$ be a outer normal vector of $\N(M)$ with $M$ rigid, then by Theorem \ref{T:Newton} either $\ep=-e_i$ or $E=\Coker(\ep)$ is an exceptional representation.
We thus have the functors $\pi_{Q_\ep}$ and $\check{\pi}_{Q_\ep}$ (see the remarks after Lemma \ref{L:ortho}).
Let $\iota_\ep : K_0(\rep(Q_\ep)) \to  K_0(\rep(Q))$ be the linear map defined by $e_i \mapsto \dv V_i$ where $V_i$ is the $\ep$-stable representation corresponding to the $i$-th vertex of $Q_\ep$.
\begin{corollary} \label{C:acyclicf} 
Let $M$ be a rigid representation of $Q$. 
Then \begin{enumerate}
	\item \hspace{-.27cm} The restriction of $F_M$ to a 0-facet is $F_{\check{\pi}_{Q_\ep}(M)}$ under the monomial transformation $\iota_\ep$;
	\item \hspace{-.27cm} The restriction of $F_M$ to a +-facet is $F_{\pi_{Q_\ep}(M)}$ under $\iota_\ep$ then times $y^{\dv h E}$,
\end{enumerate}
where $\ep$ is the outer normal vector of this facet with $\Coker(\ep)=E$.
\end{corollary}
\begin{proof} 
If $\hom(\ep,M)=0$, then $t(M)=0$ and $\tc(M)$ is the kernel of the universal homomorphism $M \to e \tau E$ by Corollary \ref{C:univhomo}.
Since $M$ is fully supported on $Q$, $\ep \neq e_i$ so $\tau E$ is a well-defined representation.
On the other hand, we recall from \cite[Proposition 3.5]{GL} that when $\hom(\ep,M)=0$, the functor $\check{p}_{\ep}: \rep(Q) \to \ss{\ep}{Q}$ is exactly given by taking the kernel of $M \to e\tau E$.
This shows that $\check{p}_{\ep}(M) = \br{\ep}^\perp(M)$ and hence $\check{\pi}_{Q_\ep}(M)=\pi_\ep(M)$.

If $\hom(\ep,M)>0$, i.e., $\ep\notin {\sf R}_0$, then Theorem \ref{T:Newton} implies that $\ep\in {\sf R}_1$ so $\ext(\ep,M)=0$.
Then $\tc(M)=M$ and $t(M)$ is the kernel of the universal homomorphism $hE \to M$ by Corollary \ref{C:univhomo}.
Moreover, the proof of Lemma \ref{L:inj} says that the universal homomorphism is injective.
On the other hand, we recall from \cite[Proposition 3.2]{GL} that when $\ext(\ep,M)=0$, the functor $p_\ep: \rep(Q) \to \ss{\ep}{Q}$ is given by taking the cokernel of the universal homomorphism $hE \hookrightarrow M$.
This shows that $p_{\ep}(M) = \br{\ep}^\perp(M)$ and hence ${\pi}_{Q_\ep}(M)=\pi_\ep(M)$.
The results then follow from Theorem \ref{T:faces}.
\end{proof}

\begin{example} Let $Q$ be the quiver $\twoone{1}{2}{3}$, and $\alpha$ be the real Schur root $(2,4,1)$.
	$$F_M=1+3y_2 + 3y_2^2 + y_2^3 + y_3 + 4y_2y_3 + 6y_2^2y_3 + 4y_2^3y_3 + y_2^4y_3  + 2y_1y_2^2y_3 + 4y_1y_2^3y_3 + 2y_1y_2^4y_3 + y_1^2y_2^4y_3.$$
Let $T$ be the unique indecomposable representation of dimension $(1,2,0)$.
The variable $\yb_1, \yb_2, \yb_3$ corresponds to the vertex $T,P_2$, and $I_3$ respectively.
Recall that $h = \hom(\epsilon,M) = f_M(\epsilon)$ and $e = \e(\epsilon,M) = \fc_M(-\epsilon)$.
\begin{center} \hspace*{-1cm} 
	\begin{tabular}{l*{7}{c}r} 
		 &$\ep$ & $h$ & $e$ & $Q_{\ep}$ & $\dv$ & $F_{\pi_\ep(M)}$ & restriction of $F_M$ &\\ \hline
		& $(2,-1,0)$ & 0 & 0 & $\Krontwo{T}{S_3}$ &  $(2,1)$  & $1+ y_3  + 2\yb_1y_3 + \yb_1^2y_3$ & $1 + y_3 + 2y_1y_2^2y_3 + y_1^2y_2^4y_3$&\\
		& $(1,0,-2)$ & 0 & 0 & $\Kronthree{I_3}{S_2}$ &  $(1,3)$  & $1 + 3y_2 + 3y_2^2 + y_2^3 + y_2^3\yb_3$ & $1+ 3y_2 + 3y_2^2  + y_2^3 + y_1^2y_2^4y_3$&\\
		& $(-1,0,0)$ & 0 & 2 & $\Kronone{S_2}{S_3}$ & $(4,1)$ & $(1+y_2)^3(1+y_3 + y_2y_3)$ & $(1+y_2)^3(1+y_3 + y_2y_3)$&\\ 
		& $(0,0,1)$ & 1 & 0 & $\Krontwo{S_1}{S_2}$ &  $(2,4)$ &  $(1+2y_2 + y_2^2 + y_1y_2^2)^2$ & $y_3(1+2y_2 + y_2^2 + y_1y_2^2)^2$ &\\ %
		& $(0,1,-1)$ & 3 & 0 & $\Krontwo{S_1}{P_2}$  &  $(2,1)$ & $1+ \yb_2  + 2y_1\yb_2 + y_1^2\yb_2$ & $y_2^3(1+ y_2y_3 + 2y_1y_2y_3 + y_1^2y_2y_3)$&  
	\end{tabular}
\end{center}
\end{example}


\section*{Acknowledgement}
The author would like to thank Hugh Thomas for encouraging him to publish this manuscript.
He would like to thank Julian K\"{u}lshammer and Bernhard Keller for the enlightening discussion on the category of twisted stalks.
He would like to particularly thank the anonymous referee for numerous corrections and suggestions (including suggesting a short section to explain the relation between the general presentations and $\tau$-tilting theory).

\bibliographystyle{amsplain}

\end{document}